\documentclass[11pt,a4paper,USenglish]{article}
\usepackage{babel}
\usepackage[T1]{fontenc}
\usepackage[utf8]{inputenc}
\usepackage{fullpage}
\usepackage[babel]{microtype}
\usepackage{csquotes}
\usepackage{authblk}
\usepackage{titlesec}
\usepackage{caption}
\usepackage{booktabs}
\usepackage[giveninits,sortcites,maxbibnames=10]{biblatex}
\usepackage{hyperref}
\usepackage{algpseudocode}
\usepackage{algorithm}
\usepackage{array}
\usepackage{subfigure}
\usepackage{tikz}
\usepackage[mathscr]{euscript}
\usepackage{enumitem}
\usepackage{amsmath}
\usepackage{mathtools}
\usepackage{amsthm}
\usepackage{thmtools}
\usepackage{dsfont}
\usepackage{amssymb}
\usepackage{accents}
\usepackage[noabbrev,capitalise,sort]{cleveref}
\usepackage{mleftright}

\hypersetup{
    pdfencoding = auto,
    colorlinks = true,
    linkcolor = blue,
    citecolor = teal,
    bookmarksnumbered = true,
    hypertexnames = false,
    pdfauthor = {Valentina Coscetti and Marco Pozza},
    pdftitle = {Numerical Approximation of the Critical Value of Eikonal Hamilton–Jacobi Equations on Networks},
    pdfkeywords = {Hamilton–Jacobi equations;numerical approximation;Mañé critical value;embedded networks;large time behavior;effective Hamiltonian},
    pdfsubject = {MSC(2020): 35R02, 65M15, 49L25, 35B40},
}

\titlespacing{\paragraph}{0pt}{1em}{.7em}

\declaretheorem[style=plain,parent=section,title=Theorem,refname={Theorem,Theorems}]{theo}
\declaretheorem[style=plain,sibling=theo,title=Proposition,refname={Proposition,Propositions}]{prop}
\declaretheorem[style=plain,sibling=theo,title=Corollary,refname={Corollary,Corollaries}]{cor}
\declaretheorem[style=definition,sibling=theo,title=Definition,refname={Definition,Definitions}]{defin}
\declaretheorem[style=remark,sibling=theo,title=Remark,refname={Remark,Remarks}]{rem}

\AddToHook{cmd/appendix/before}{\crefalias{chapter}{appendix}\crefalias{section}{appendix}\crefalias{subsection}{appendix}}

\newlist{Hassum}{enumerate}{1}
\setlist[Hassum]{label=\textbf{(H\arabic*)},ref=\textnormal{\textbf{(H\arabic*)}},font=\normalfont}
\newlist{myenum}{enumerate}{1}
\setlist[myenum]{label=\textbf{\roman*)},ref=\textnormal{\textbf{(\roman*)}},font=\normalfont}
\newlist{mylist}{itemize}{1}
\setlist[mylist]{label=\textbullet,font=\normalfont}

\crefname{equation}{}{}
\crefname{Hassumi}{condition}{conditions}
\crefname{line}{line}{lines}

\usetikzlibrary{decorations.markings}

\DeclareMathOperator{\diam}{diam}
\DeclareMathOperator*{\amax}{arg\,max}
\DeclareMathOperator*{\amin}{arg\,min}

\newcommand{\Acal}{\mathcal{A}}
\newcommand{\wtAcal}{\widetilde{\Acal}}
\newcommand{\Drm}{\mathrm{D}}
\newcommand{\Ecal}{\mathcal{E}}
\newcommand{\wtH}{\widetilde{H}}
\newcommand{\Hcal}{\mathcal{H}}
\newcommand{\wtHcal}{\widetilde{\Hcal}}
\newcommand{\Mcal}{\mathcal{M}}
\newcommand{\Nds}{\mathds{N}}
\newcommand{\Rds}{\mathds{R}}
\newcommand{\Scal}{\mathcal{S}}
\newcommand{\Sscr}{\mathscr{S}}
\newcommand{\wtS}{\widetilde{S}}
\newcommand{\Tds}{\mathds{T}}
\newcommand{\Tscr}{\mathscr{T}}
\newcommand{\Vbf}{\mathbf{V}}
\newcommand{\ovc}{\overline{c}}
\newcommand{\unc}{\underline{c}}
\newcommand{\utc}{\undertilde{c}}
\newcommand{\wtc}{\widetilde{c}}
\newcommand{\drm}{\mathrm{d}}
\newcommand{\ovw}{\overline{w}}
\newcommand{\unw}{\underline{w}}
\newcommand{\ovx}{\overline{x}}
\newcommand{\ovy}{\overline{y}}
\newcommand{\wtgamma}{\widetilde{\gamma}}
\newcommand{\ovphi}{\overline{\phi}}
\newcommand{\unphi}{\underline{\phi}}
\newcommand{\myemail}[2][]{\textsuperscript{#1}\href{mailto:#2}{\texttt{#2}}}
\newcommand{\hrefc}{\textup{(}\hyperref[eq:globeik]{\ensuremath{\Hcal}\textup{J}\ensuremath{c}}\textup{)}}

\title{Numerical Approximation of the Critical Value of Eikonal Hamilton--Jacobi Equations on Networks}
\author[1]{Valentina Coscetti}
\author[2]{Marco Pozza}
\affil[1]{Sapienza, University of Rome, Italy. \textit{Email address:} \myemail{valentina.coscetti@uniroma1.it}}
\affil[2]{Link Campus University, Rome, Italy. \textit{Email address:} \myemail{m.pozza@unilink.it}}
\date{}

\begin{document}

    \maketitle

    \begin{abstract}
        The \emph{critical value} of an eikonal equation is the unique value of a parameter for which the equation admits solutions and is deeply related to the effective Hamiltonian of a corresponding homogenization problem. We study approximation strategies for the critical value of eikonal equations posed on networks. They are based on the large time behavior of corresponding time-dependent Hamilton--Jacobi equations. We provide error estimates and some numerical tests, showing the performance and the convergence properties of the proposed algorithms.
    \end{abstract}

    \paragraph{2020 Mathematics Subject Classification:} 35R02, 65M15, 49L25, 35B40.

    \paragraph{Keywords:} Hamilton--Jacobi equations, numerical approximation, Mañé critical value, embedded networks, large time behavior, effective Hamiltonian.

    \section{Introduction}

    Given a Hamiltonian $H:\Tds^N\times\Rds^N\to\Rds$, where $\Tds^N$ is the $N$-dimensional torus, the \emph{Mañé critical value} of the eikonal equation defined by $H$ is the unique value $c$ such that
    \begin{equation*}
        H(x,\Drm u)=c,\qquad x\in\Tds^N,
    \end{equation*}
    admits solutions. The critical value is also deeply connected to the homogenization of Hamilton--Jacobi equations, which starting from the seminal paper~\cite{LionsPapanicolaouVaradhan87} has received a considerable interest both from theoretical and from an applied viewpoint, see for example~\cite{Capuzzo-DolcettaIshii01,MajdaSouganidis94,ContrerasIturriagaSiconolfi14,Evans92,NamahRoquejoffre00}. It then comes as no surprise that many works, such as~\cite{GomesOberman04,AchdouCamilliCapuzzo-Dolcetta08,CamilliCapuzzoDolcettaGomes07,CacaceCamilli16}, are devoted to the approximation of the critical value for problems posed on $\Tds^N$.

    The aim of this paper is to approximate the critical value of eikonal equations posed on \emph{networks}. With network we mean a connected finite graph $\Gamma$ embedded in $\Rds^N$ with vertices linked by regular simple curves $\gamma$, called arcs of the network. A Hamiltonian on $\Gamma$ is a collection of Hamiltonians $\{H_\gamma\}$ indexed by the arcs, depending on state and momentum variables, with the crucial feature that Hamiltonians associated to arcs possessing different support are totally unrelated. Over the last years there has been an increasing interest in the study of this type of problem, leading to many valuable contributions. Among these we cite~\cite{BarlesChasseigne24,CamilliSchieborn12,ImbertMonneau17,ImbertMonneauZidani12,LionsSouganidis17,AchdouTchou15}.

    The eikonal problem on $\Gamma$, as defined in~\cite{SiconolfiSorrentino18}, is to find a continuous function $u:\Gamma\to\Rds$ such that $u\circ\gamma$ is a viscosity solution to
    \begin{equation}\label{eq:introeik}
        H_\gamma(s,\partial U)=c,
    \end{equation}
    for any arc $\gamma$, where $c\in\Rds$ is a fixed parameter. As in the torus case, the critical value of the eikonal equation is the unique value of the parameter $c$ for which~\eqref{eq:introeik} admits solutions and is deeply related to the so-called \emph{effective Hamiltonian} of a corresponding homogenization problem, see~\cite{PozzaSiconolfiSorrentino24}.

    In general, \cref{eq:introeik} cannot be solved explicitly. Moreover, to design a numerical scheme for such an equation is a very difficult task since it requires the approximation of a first order Hamilton--Jacobi problem in which both the solution $u$ and the critical value $c$ are unknown. Indeed, there are few papers devoted to the numerical approximation of solutions to Hamilton--Jacobi equations on networks and most of them are focused on time-dependent problems as described in~\cite{ImbertMonneau17,Siconolfi22}, which are, from the numerical point of view, simpler than the eikonal one. Let us mention the finite difference based methods outlined in~\cite{CamilliCarliniMarchi18,CostesequeLebacqueMonneau14} and the semi-Lagrangian schemes in~\cite{CarliniFestaForcadel20,CarliniSiconolfi23,CamilliFestaSchieborn13}.

    Recently, in~\cite{Pozza25_2}, it has been shown that the large time behavior of solutions to time-dependent Hamilton--Jacobi equations can be exploited to obtain solutions to a corresponding eikonal problem. Accordingly, our idea is to use the large time behavior of such solutions to approximate the critical value.

    We provide two numerical schemes for the approximation of the critical value. The first one, \cref{basicalgo}, uses an approach already employed in~\cite{AchdouCamilliCapuzzo-Dolcetta08} for problems posed on $\Tds^N$ and allows an a priori error estimate. For \cref{iteralgo} we provide a convergence theorem and our simulations show that it usually achieves the same results of \cref{basicalgo} with fewer steps.

    The paper is organized as follows: \cref{netsec} provides some basic facts about networks and Hamiltonians defined on them, and gives our main assumptions. In \cref{HJsec} we introduce the time-dependent equation and its large time behavior, which will provide the theoretical basis for our algorithms. \Cref{algosec} is about the analysis of our approximation strategies for the critical value. Finally, in \cref{simsec}, we present the numerical simulations of our algorithms. \Cref{ltbexp} is devoted to the proof of a \namecref{ltb} which extends the large time behavior analysis performed in~\cite{Pozza25_2} to our framework.

    \paragraph{Acknowledgments.} The first author was partially supported by Italian Ministry of Instruction, University and Research (MIUR) (PRIN Project2022238YY5, ``Optimal control problems: analysis, approximation'') by INdAM–GNCS Project (CUP$\_$ E53C24001950001, ``Metodi avanzati per problemi di Mean Field Games ed applicazioni''), by Avvio alla Ricerca (CUP$\_$ B83C24006550001, ``Numerical schemes for Hamilton-Jacobi equations and Mean Field Games problems'') and by European Union - Next Generation EU, Missione 4, Componente 1, CUP$\_$ B53C23002010006. The second author is a member of the INdAM research group GNAMPA.

    \section{Networks}\label{netsec}

    Here we describe our framework, namely what is a network, a Hamiltonian defined on it and some related concepts useful for our study.

    \subsection{Basic Definitions}

    We fix a dimension $N$ and $\Rds^N$ as ambient space. An \emph{embedded network}, or \emph{continuous graph}, is a subset $\Gamma\subset\Rds^N$ of the form
    \begin{equation*}
        \Gamma=\bigcup_{\gamma\in\Ecal}\gamma([0,|\gamma|])\subset\Rds^N,
    \end{equation*}
    where $\Ecal$ is a finite collection of regular (i.e., $C^1$ with non-vanishing derivative) simple oriented curves, called \emph{arcs} of the network, with Euclidean length $|\gamma|$ and parameterized by arc length in $[0,|\gamma|]$ (i.e., $|\dot\gamma|\equiv1$ for any $\gamma\in\Ecal$). Note that we are also assuming existence of one-sided derivatives at the endpoints 0 and $|\gamma|$. We stress out that a regular change of parameters does not affect our results.

    On the support of any arc $\gamma$, we also consider the inverse parametrization defined as
    \begin{equation*}
        \widetilde\gamma(s)\coloneqq\gamma(|\gamma|-s),\qquad \text{for }s\in[0,|\gamma|].
    \end{equation*}
    We call $\widetilde\gamma$ the \emph{inverse arc} of $\gamma$. We assume
    \begin{equation}\label{eq:nosovrap}
        \gamma((0,|\gamma|))\cap\gamma'\mleft(\mleft[0,\mleft|\gamma'\mright|\mright]\mright)=\emptyset,\qquad \text{whenever }\gamma'\ne\gamma,\widetilde\gamma.
    \end{equation}

    We call \emph{vertices} the initial and terminal points of the arcs, and denote by $\Vbf$ the sets of all such vertices. It follows from~\eqref{eq:nosovrap} that vertices are the only points where arcs with different support intersect and, in particular,
    \begin{equation*}
        \gamma((0,|\gamma|))\cap\Vbf=\emptyset,\qquad \text{for any }\gamma\in\Ecal.
    \end{equation*}

    We set, for each $x\in\Vbf$,
    \begin{equation*}
        \Gamma_x\coloneqq\{\gamma\in\Ecal:\gamma(0)=x\}.
    \end{equation*}

    We assume that the network is connected, namely given two vertices there is a finite concatenation of arcs linking them. A \emph{loop} is an arc with initial and final point coinciding. The unique restriction we require on the geometry of $\Gamma$ is
    \begin{mylist}
        \item $\Ecal$ does not contain loops.
    \end{mylist}
    This condition is due to the fact that in the known literature about time-dependent Hamilton--Jacobi equations on networks no loops are admitted, see, e.g., \cite{ImbertMonneau17,Siconolfi22,LionsSouganidis17}. We believe that our approach, the same used in~\cite{Siconolfi22}, can also include the presence of loops, but this should require nontrivial adjustments like the addition of a periodic condition on loops.

    The network $\Gamma$ inherits a geodesic distance, denoted with $d_\Gamma$, from the Euclidean metric of $\Rds^N$. It is clear that given $x$, $y$ in $\Gamma$ there is at least a geodesic linking them. The geodesic distance is in addition equivalent to the Euclidean one. We will denote with $\diam(\Gamma)$ the diameter of $\Gamma$ with respect to the geodesic distance.

    We also consider a differential structure on the network by defining the \emph{tangent bundle} of $\Gamma$, $T\Gamma$ in symbols, as the set made up by the $(x,q)\in\Gamma\times\Rds^N$ with $q$ of the form
    \begin{equation*}
        q=\lambda\dot\gamma(s),\qquad \text{if $x=\gamma(s)$, $s\in[0,|\gamma|]$, with $\lambda\in\Rds$}.
    \end{equation*}
    Note that $\dot\gamma(s)$ is univocally determined, up to a sign, if $x\in\Gamma\setminus\Vbf$ or in other words if $s\notin\{0,|\gamma|\}$.

    For the sake of simplicity, by curve we mean throughout the paper an \emph{absolutely continuous} curve. We point out that the pair $\mleft(\xi,\dot\xi\mright)$, where $\xi$ is a curve in $\Gamma$, is naturally contained in $T\Gamma$.

    \subsection{Hamiltonians on \texorpdfstring{$\Gamma$}{Γ}}

    A Hamiltonian on $\Gamma$ is a collection of Hamiltonians $\Hcal\coloneqq\{H_\gamma\}_{\gamma\in\Ecal}$, where
    \begin{alignat*}{2}
        H_\gamma:[0,|\gamma|]\times\Rds&&\:\longrightarrow\:&\Rds\\
        (s,\mu)&&\:\longmapsto\:&H_\gamma(s,\mu),
    \end{alignat*}
    satisfying
    \begin{equation*}
        H_{\wtgamma}(s,\mu)=H_\gamma(|\gamma|-s,-\mu),\qquad \text{for any arc }\gamma.
    \end{equation*}
    We emphasize that, apart the above compatibility condition, the Hamiltonians $H_\gamma$ are \emph{unrelated}.

    We require each $H_\gamma$ to be:
    \begin{Hassum}
        \item\label{condcont} continuous in both arguments;
        \item Lipschitz continuous in $s$ for any $\mu\in\Rds$;
        \item\label{condsuplin} superlinearly coercive in the momentum variable, uniformly in $s$, i.e.,
        \begin{equation*}
            \lim_{|\mu|\to\infty}\inf_{s\in [0,|\gamma|]}\frac{H_\gamma(s,\mu)}{|\mu|}=\infty,\qquad \text{for any $\gamma\in\Ecal$};
        \end{equation*}
        \item\label{condconv} convex in $\mu$;
        \item\label{condqconv} strictly quasiconvex in $\mu$, which means that, for any $s\in[0,|\gamma|]$, $\mu,\mu'\in\Rds$ and $\rho\in(0,1)$,
        \begin{equation*}
            H_\gamma\mleft(s,\rho\mu+(1-\rho)\mu'\mright)<\max\mleft\{H_\gamma(s,\mu),H_\gamma\mleft(s,\mu'\mright)\mright\}.
        \end{equation*}
    \end{Hassum}
    Note that if $H_\gamma$ is strictly convex in the momentum variable, then it satisfies both \labelcref{condconv,condqconv}.\\
    Under \cref{condcont,condsuplin,condconv} it is natural to define, for any $\gamma\in\Ecal$, the \emph{Lagrangian} corresponding to $H_\gamma$ as
    \begin{equation*}
        L_\gamma(s,\lambda)\coloneqq\sup_{\mu\in\Rds}(\lambda\mu-H_\gamma(s,\mu)),
    \end{equation*}
    where the supremum is actually achieved thanks to~\ref{condsuplin}. We have, for each $\lambda\in\Rds$ and $s\in[0,|\gamma|]$,
    \begin{equation}\label{eq:lagcomp}
        L_\gamma(s,\lambda)=L_{\widetilde\gamma}(|\gamma|-s,-\lambda).
    \end{equation}
    Moreover, the Lagrangians $L_\gamma$ are Lipschitz continuous in the first variable, and convex and superlinear in the second one. We will define later on a Lagrangian defined on the whole network, assuming suitable gluing conditions on the $L_\gamma$.

    \section{Hamilton--Jacobi Equations on Networks}\label{HJsec}

    In this \lcnamecref{HJsec} we briefly recall some results about Hamilton--Jacobi equations posed on networks which will form the theoretical basis of our algorithms.

    \subsection{Time-Dependent HJ Equations}

    We focus on the time-dependent problem, thoroughly analyzed in~\cite{Siconolfi22,ImbertMonneau17}, of the form
    \begin{equation}\label{eq:globteik}\tag{\ensuremath{\Hcal}JE}
        \partial_t v(x,t)+\Hcal(x,\Drm_x v)=0,\qquad \text{on }\Gamma\times(0,\infty).
    \end{equation}
    This notation synthetically indicates the family (for $\gamma$ varying in $\Ecal$) of Hamilton--Jacobi equations
    \begin{equation}\label{eq:teikg}\tag{HJ\textsubscript{\ensuremath{\gamma}}E}
        \partial_t U(s,t)+H_\gamma(s,\partial_s U(s,t))=0,\qquad \text{on }(0,|\gamma|)\times(0,\infty).
    \end{equation}
    It has been established that to get existence and uniqueness of solutions the equations~\cref{eq:teikg} must be coupled with an additional condition on the one-dimensional interfaces
    \begin{equation*}
        \mleft\{(x,t):t\in\Rds^+\mright\}\qquad \text{with $x\in\Vbf$}.
    \end{equation*}
    Following~\cite{ImbertMonneau17}, we set
    \begin{equation*}
        a_\gamma\coloneqq\max_{s\in[0,|\gamma|]}\min_{\mu\in\Rds}H_\gamma(s,\mu),\quad \text{for each }\gamma\in\Ecal,
    \end{equation*}
    and call \emph{flux limiter} any function $x\mapsto c_x$ from $\Vbf$ to $\Rds$ satisfying
    \begin{equation*}
        c_x\ge\max_{\gamma\in\Gamma_x}a_\gamma,\qquad \text{for any }x\in\Vbf.
    \end{equation*}
    Hereafter we will only consider the minimal flux limiter $c_x$, namely
    \begin{equation*}
        c_x\coloneqq\max_{\gamma\in\Gamma_x}a_\gamma,\qquad \text{for any }x\in\Vbf.
    \end{equation*}
    The flux limiter, roughly speaking, is a choice of a constant on each vertex which bound from above the time derivative of any subsolutions on it. We refer to~\cite{Siconolfi22} for the definition of solution to~\cref{eq:globteik} and the role of the flux limiter in it.

    Next we define a global Lagrangian $L:T\Gamma\to\Rds$ by gluing together the $L_\gamma$ and taking into account the flux limiter $c_x$:
    \begin{mylist}
        \item if $x=\gamma(s)$ for some $\gamma\in\Ecal$ and $s\in(0,1)$ then
        \begin{equation*}
            L(x,q)\coloneqq L_\gamma\mleft(s,\frac{q\dot\gamma(s)}{|\dot\gamma(s)|^2}\mright);
        \end{equation*}
        \item if $x\in\Vbf$ and $q\ne0$ then
        \begin{equation*}
            L(x,q)\coloneqq\min L_\gamma\mleft(0,\frac{q\dot\gamma(0)}{|\dot\gamma(0)|^2}\mright),
        \end{equation*}
        where the minimum is taken over the $\gamma\in\Gamma_x$ with $\dot\gamma(0)$ parallel to $q$;
        \item if $x\in\Vbf$ and $q=0$ then
        \begin{equation*}
            L(x,q)\coloneqq-c_x.
        \end{equation*}
    \end{mylist}
    We notice that thanks to~\eqref{eq:lagcomp} $L$ is a well-defined function in $T\Gamma$.

    \begin{theo}
        \emph{\cite[Theorem 6.7]{PozzaSiconolfi23}} Given an initial datum $\phi\in C(\Gamma)$, the function
        \begin{equation}\label{eq:vft}
            (\Scal(t)\phi)(x)\coloneqq\min\mleft\{\phi(\xi(0))+\int_0^t L\mleft(\xi,\dot\xi\mright)\drm\tau: \text{$\xi$ is a curve with }\xi(t)=x\mright\}.
        \end{equation}
        is the unique solution to~\cref{eq:globteik} with $\Scal(0)\phi=\phi$ and flux limiter $c_x$.
    \end{theo}

    We stress out that there exists an optimal curve for $(\Scal(t)\phi)(x)$. Furthermore, the family of operators $\{\Scal(t)\}_{t\in\Rds^+}$ form a semigroup whose main properties are summarized below.

    \begin{prop}\label{Scalprop}
        \emph{\cite[Proposition 4.13]{Pozza25_2}}
        \begin{myenum}
            \item\emph{(Semigroup property)} For any $t,t'\in\Rds^+$ we have $\Scal(t+t')=\Scal(t)\circ\Scal(t')$;
            \item\emph{(Monotonicity property)} for every $\phi_1,\phi_2\in C(\Gamma)$ such that $\phi_1\le\phi_2$ in $\Gamma$
            \begin{equation*}
                \Scal(t)\phi_1\le\Scal(t)\phi_2,\qquad \text{for any }t\in\Rds^+;
            \end{equation*}
            \item for any $\phi\in C(\Gamma)$, $t\in\Rds^+$ and $a\in\Rds$, $\Scal(t)(\phi+a)=\Scal(t)\phi+a$.
        \end{myenum}
    \end{prop}

    \begin{theo}\label{liptsol}
        \emph{\cite[Theorem 7.4]{Siconolfi22}} If the initial datum $\phi$ is Lipschitz continuous, then $(x,t)\mapsto(\Scal(t)\phi)(x)$ is Lipschitz continuous in $\Gamma\times\Rds^+$.
    \end{theo}

    \subsection{Approximation of the Critical Value}

    We consider the eikonal equation
    \begin{equation}\label{eq:globeik}\tag{\ensuremath{\Hcal}J\ensuremath{a}}
        \Hcal(x,\Drm u)=a,\qquad \text{on }\Gamma,
    \end{equation}
    namely the family of Hamilton--Jacobi equations
    \begin{equation*}
        H_\gamma(s,\partial U)=a,\qquad \text{on }[0,|\gamma|],
    \end{equation*}
    where $a\in\Rds$, for $\gamma$ varying in $\Ecal$ and coupled with a state-constraint-type boundary condition at the vertices. The \emph{critical value}, or \emph{Mañé critical value}, is characterized in~\cite{SiconolfiSorrentino18,Pozza25} as the unique value
    \begin{equation}\label{eq:critvaldef}
        c\ge a_0\coloneqq\max_{\gamma\in\Ecal}a_\gamma
    \end{equation}
    for which \hrefc{} (namely the equation~\eqref{eq:globeik} with $a=c$) admits solutions.

    It has been shown in~\cite{Pozza25_2} that $\Scal(t)\phi+ct$ converges to a solution to~\hrefc{} as $t\to\infty$. We will use this fact to approximate the critical value.

    The large time behavior of the solutions to~\cref{eq:globteik} is mainly due to the relation between the optimal curves of~\cref{eq:vft} and the so called Aubry set. More in detail, let $\sigma_c$ be a real map on the tangent bundle $T\Gamma$ such that:
    \begin{mylist}
        \item if $x=\gamma(s)$ for some $\gamma\in\Ecal$ and $s\in(0,|\gamma|)$ then
        \begin{equation*}
            \sigma_c(x,q)\coloneqq\max\{\mu q\dot\gamma(s):\mu\in\Rds,\,H_\gamma(s,\mu)=c\};
        \end{equation*}
        \item if $x\in\Vbf$ then
        \begin{equation*}
            \sigma_c(x,q)\coloneqq\min\max\{\mu q\dot\gamma(0):\mu\in\Rds,\,H_\gamma(0,\mu)=c\},
        \end{equation*}
        where the minimum is taken over the $\gamma\in\Gamma_x$ with $\dot\gamma(0)$ parallel to $q$.
    \end{mylist}

    \begin{defin}
        We call \emph{Aubry set} on $\Gamma$, the closed set $\Acal_\Gamma$ made up of
        \begin{mylist}
            \item the $x\in\Gamma$ incident to a closed curve $\xi:[0,T]\to\Gamma$ with $\int_0^T\sigma_c\mleft(\xi,\dot\xi\mright)d\tau=0$ and a.e.\ non-vanishing derivative;
            \item the $x=\gamma(s)$ with $\gamma\in\Ecal$ and $s\in[0,1]$ such that the set $\{\mu\in\Rds:\,H_\gamma(s,\mu)=c\}$ is a singleton.
        \end{mylist}
    \end{defin}

    Setting the semidistance
    \begin{equation*}
        S_c(y,x)\coloneqq\min\mleft\{\int_0^T\sigma_c\mleft(\xi,\dot\xi\mright)\drm\tau: \text{$\xi:[0,T]\to\Gamma$ is a curve from $y$ to $x$}\mright\}
    \end{equation*}
    we further have

    \begin{theo}\label{maxsubsol}
        \emph{\cite[Theorem 4.8 and Corollary 4.10]{Pozza25}} Let $\Gamma'$ be a closed subset of $\Gamma$, $w:\Gamma'\to\Rds$ be a continuous function and define
        \begin{equation*}
            u(x)\coloneqq\min_{y\in\Gamma'}(w(y)+S_c(y,x)),\qquad \text{for }x\in\Gamma.
        \end{equation*}
        Then $u$ is the maximal subsolution to \hrefc{} not exceeding $w$ on $\Gamma'$ and a solution in $\Gamma\setminus(\Gamma'\setminus\Acal_\Gamma)$. If $w$ is a subsolution $u$ agrees with it on $\Gamma'$.
    \end{theo}

    As pointed out in~\cite{Pozza25_2}, the convergence of $\Scal(t)\phi+ct$ is influenced by the flux limiter. We account for this by extending, in a certain sense, the Aubry set.

    \begin{defin}
        We call \emph{extended Aubry set} on $\Gamma$, the closed set $\wtAcal_\Gamma$ made up by
        \begin{mylist}
            \item the $x\in\Acal_\Gamma$;
            \item the $x\in\Vbf$ such that $c_x=c$.
        \end{mylist}
    \end{defin}

    The large time behavior of solutions to~\cref{eq:globteik} in our setting is described by the next \namecref{ltb}, whose proof is given in \cref{ltbexp}.

    \begin{theo}\label{ltb}
        Given $\phi\in C(\Gamma)$, the function $\Scal(t)\phi+ct$ uniformly converges, as $t$ goes to $\infty$, to
        \begin{equation}\label{eq:ltb.1}
            u(x)\coloneqq\min_{y\in\wtAcal_\Gamma}\mleft(\min_{z\in\Gamma}(\phi(z)+S_c(z,y))+S_c(y,x)\mright),\qquad \text{for }x\in\Gamma.
        \end{equation}
    \end{theo}

    The limit function~\cref{eq:ltb.1} is a solution to \hrefc{} in $\Gamma\setminus\mleft(\wtAcal_\Gamma\setminus\Acal_\Gamma\mright)$ agreeing with
    \begin{equation*}
        w(x)\coloneqq\min\limits_{y\in\Gamma}(\phi(y)+S_c(y,x)),\qquad \text{for }x\in\Gamma,
    \end{equation*}
    on $\wtAcal_\Gamma$, as well as the maximal subsolution in $\Gamma$ equal to $w$ on $\wtAcal_\Gamma$, see \cref{maxsubsol}.

    Let us now list some consequences of \cref{ltb} which will provide the theoretical basis for the algorithms in \cref{algosec}.

    \begin{theo}\label{basictheo}
        Given a Lipschitz continuous function $\phi$, we have that
        \begin{equation}\label{eq:basictheo.1}
            \mleft|\frac{\phi(x)-(\Scal(t)\phi)(x)}t-c\mright|\le2\ell\diam(\Gamma)\frac1t,\qquad \text{for any }(x,t)\in\Gamma\times(0,\infty),
        \end{equation}
        where $\ell$ is the Lipschitz constant of $(x,t)\mapsto(\Scal(t)\phi)(x)$.
    \end{theo}
    \begin{proof}
        Let $u$ be a solution to \hrefc{} in $\Gamma\setminus\wtAcal_\Gamma$. We know from \cref{weakltb,Scalprop} that
        \begin{equation*}
            (\Scal(t)\phi)(x)+ct\le u(x)+\max_{y\in\Gamma}(\phi(y)-u(y)),\qquad \text{for any }(x,t)\in\Gamma\times\Rds^+.
        \end{equation*}
        It follows that there is an $x_1\in\Gamma$ such that
        \begin{equation}\label{eq:basictheo1}
            (\Scal(t)\phi)(x_1)+ct\le\phi(x_1),\qquad \text{for any }t\in\Rds^+,
        \end{equation}
        thus \cref{liptsol} yields, for any $(x,t)\in\Gamma\times\Rds^+$,
        \begin{equation}\label{eq:basictheo2}
            \begin{aligned}
                \phi(x)-(\Scal(t)\phi)(x)-ct\ge\;&\phi(x)-(\Scal(t)\phi)(x)-ct-\phi(x_1)+(\Scal(t)\phi)(x_1)+ct\\
                \ge\;&-2\ell d_\Gamma(x_1,x)\ge-2\ell\diam(\Gamma).
            \end{aligned}
        \end{equation}
        Exploiting once again \cref{weakltb,Scalprop} we get
        \begin{equation*}
            (\Scal(t)\phi)(x)+ct\ge u(x)+\min_{y\in\Gamma}(\phi(y)-u(y)),\qquad \text{for any }(x,t)\in\Gamma\times\Rds^+,
        \end{equation*}
        thereby there is an $x_2\in\Gamma$ such that
        \begin{equation}\label{eq:basictheo3}
            (\Scal(t)\phi)(x_2)+ct\ge\phi(x_2),\qquad \text{for any }t\in\Rds^+,
        \end{equation}
        and, as a consequence,
        \begin{equation*}
            \phi(x)-(\Scal(t)\phi)(x)-ct\le2\ell\diam(\Gamma),\qquad \text{for any }(x,t)\in\Gamma\times\Rds^+.
        \end{equation*}
        This, together with~\eqref{eq:basictheo2}, implies~\eqref{eq:basictheo.1} and concludes our proof.
    \end{proof}

    The bounds given below are a simple consequence of~\eqref{eq:basictheo1} and~\eqref{eq:basictheo3}.

    \begin{cor}\label{basicbounds}
        Given a Lipschitz continuous function $\phi$, we have that
        \begin{equation*}
            \min_{x\in\Gamma}\frac{\phi(x)-(\Scal(t)\phi)(x)}t\le c\le\max_{x\in\Gamma}\frac{\phi(x)-(\Scal(t)\phi)(x)}t,\qquad \text{for any }t\in(0,\infty).
        \end{equation*}
    \end{cor}

    Next we provide a different approach to the approximation of the critical value. The main difference between \cref{basictheo} and the \namecref{itertheo} below is that the former provides an estimate for the rate of convergence of $(\phi-\Scal(t)\phi)/t$ to $c$ as $t$ positively diverges, while the latter exploits the convergence of $\Scal(t)\phi+ct$ to a solution to \hrefc{} to define two sequences converging monotonically to $c$.

    \begin{theo}\label{itertheo}
        Given $\phi\in C(\Gamma)$ and $T>0$, we recursively set $\phi_0\coloneqq\phi$, $\phi_{k+1}\coloneqq\Scal(T)\phi_k$,
        \begin{equation*}
            \wtc_k\coloneqq\max_{x\in\Gamma}\frac{\phi_{k-1}(x)-\phi_k(x)}T\qquad \text{and}\qquad\utc_k\coloneqq\min_{x\in\Gamma}\frac{\phi_{k-1}(x)-\phi_k(x)}T.
        \end{equation*}
        The sequences $\{\wtc_k\}_{k\in\Nds}$ and $\{\utc_k\}_{k\in\Nds}$ are nonincreasing and nondecreasing, respectively, and both converge to $c$.
    \end{theo}
    \begin{proof}
        Fixed $k\ge1$, let $x_1,y_1\in\Gamma$ and $\xi_1:[0,(k+1)T]\to\Gamma$ be such that
        \begin{equation*}
            \max_{x\in\Gamma}(\phi_k(x)-\phi_{k+1}(x))=\phi_k(x_1)-\phi_{k+1}(x_1)=(\Scal(T)\phi_{k-1})(x_1)-(\Scal(T)\phi_k)(x_1)
        \end{equation*}
        and
        \begin{equation*}
            (\Scal(T)\phi_k)(x_1)=\phi_k(y_1)+\int_0^T L\mleft(\xi_1,\dot\xi_1\mright)\drm\tau.
        \end{equation*}
        Then
        \begin{equation}\label{eq:itertheo1}
            \begin{aligned}
                \wtc_{k+1}T=\,&\max_{x\in\Gamma}(\phi_k(x)-\phi_{k+1}(x))\\
                \le\,&\phi_{k-1}(y_1)+\int_0^T L\mleft(\xi_1,\dot\xi_1\mright)\drm\tau-\phi_k(y_1)-\int_0^T L\mleft(\xi_1,\dot\xi_1\mright)\drm\tau\le\wtc_k T.
            \end{aligned}
        \end{equation}
        Similarly, we set $x_2,y_2\in\Gamma$ and $\xi_2:[0,(k+1)T]\to\Gamma$ such that
        \begin{equation*}
            \min_{x\in\Gamma}(\phi_k(x)-\phi_{k+1}(x))=\phi_k(x_2)-\phi_{k+1}(x_2)=(\Scal(T)\phi_{k-1})(x_2)-(\Scal(T)\phi_k)(x_2)
        \end{equation*}
        and
        \begin{equation*}
            (\Scal(T)\phi_{k-1})(x_2)=\phi_{k-1}(y_2)+\int_0^T L\mleft(\xi_2,\dot\xi_2\mright)\drm\tau.
        \end{equation*}
        Hence,
        \begin{equation}\label{eq:itertheo2}
            \begin{aligned}
                \utc_{k+1}T=\,&\min_{x\in\Gamma}(\phi_k(x)-\phi_{k+1}(x))\\
                \ge\,&\phi_{k-1}(y_2)+\int_0^T L\mleft(\xi_2,\dot\xi_2\mright)\drm\tau-\phi_k(y_2)-\int_0^T L\mleft(\xi_2,\dot\xi_2\mright)\drm\tau\ge\utc_k T.
            \end{aligned}
        \end{equation}
        Inequalities~\eqref{eq:itertheo1} and~\eqref{eq:itertheo2} show that the sequences $\{\wtc_k\}$ and $\{\utc_k\}$ are nonincreasing and nondecreasing, respectively. Finally, we exploit \cref{ltb} to obtain
        \begin{align*}
            \lim_{k\to\infty}\wtc_k=\;&\lim_{k\to\infty}\max_{x\in\Gamma}\frac{\phi_{k-1}(x)-\phi_k(x)}T\\
            =\;&\lim_{k\to\infty}\max_{x\in\Gamma}\frac{(\Scal((k-1)T)\phi)(x)+c(k-1)T-(\Scal(kT)\phi)(x)-ckT}T+c=c,
        \end{align*}
        i.e., $\{\wtc_k\}$ converges to $c$. Similarly, we can prove that $\{\utc_k\}$ converges to $c$.
    \end{proof}

    Following \cref{itertheo}, we get that $\utc_k\le c\le\wtc_k$ for any $k\in\Nds$, which implies the next result.

    \begin{cor}\label{iterapprox}
        Under the same assumptions of \cref{itertheo}, we define for every $k\in\Nds$
        \begin{equation}\label{eq:iterapprox.1}
            c_k\coloneqq\frac{\wtc_k+\utc_k}2\qquad \text{and}\qquad\varepsilon_k\coloneqq\frac{\wtc_k-\utc_k}2.
        \end{equation}
        We have that $|c_k-c|\le\varepsilon_k$ for any $k\in\Nds$. In particular $\{\varepsilon_k\}_{k\in\Nds}$ is a nonincreasing infinitesimal sequence and $c_k$ converges to $c$.
    \end{cor}

    \section{Algorithms for the Critical Value}\label{algosec}

    In this \lcnamecref{algosec} we present two algorithms for the approximation of the critical value.

    \subsection{Numerical Approximation of Time-Dependent HJ Equations}

    Our algorithms require a method to approximate solutions to~\eqref{eq:globteik}, since they rely on the large time behavior of said solutions.

    We employ the semi-Lagrangian scheme presented in~\cite{CarliniSiconolfi23} to approximate the solutions to~\eqref{eq:globteik}. Rather than directly discretizing equation~\eqref{eq:globteik}, the scheme addresses a modified problem that involves modified Hamiltonians $\wtH_\gamma$ obtained from the $H_\gamma$ through the procedure described in~\cite[Appendix A]{CarliniSiconolfi23}. This construction depends on the preliminary choice of a compact interval $I$ of the momentum variable $\mu$ within which the modified Hamiltonians coincide with the original ones:
    \begin{equation*}
        \wtH_\gamma(s,\mu)=H_\gamma(s,\mu),\qquad \text{for every }\gamma\in\Ecal,\,s\in[0,|\gamma|],\,\mu\in I.
    \end{equation*}
    The key advantage of this approach is the existence of a positive constant $\beta_0$, depending on $I$, such that the modified Lagrangians
    \begin{equation*}
        \widetilde L_\gamma(s,\lambda)\coloneqq\max_{\mu\in\Rds}\mleft(\lambda\mu-\wtH_\gamma(s,\mu)\mright),\qquad \text{for }\gamma\in\Ecal,\,s\in [0,|\gamma|],\,\lambda\in\Rds,
    \end{equation*}
    are equal to $+\infty$ when $\lambda$ is outside $[-\beta_0,\beta_0]$, for all $s\in[0,|\gamma|]$, $\gamma\in\Ecal$, and are Lipschitz continuous in $[-\beta_0,\beta_0]$. Setting $\wtHcal\coloneqq\{\wtH_\gamma\}_{\gamma\in\Ecal}$, the modified problem is
    \begin{equation}\tag{$\wtHcal$Jmod}\label{eq:globteikmod}
        \partial_t u(x,t)+\wtHcal(x,D_x u)=0,\qquad \text{on }\Gamma\times(0,\infty).
    \end{equation}
    This reformulation is justified by the fact that the solution to~\eqref{eq:globteikmod} coincides with the solution to~\eqref{eq:globteik} (see~\cite{CarliniSiconolfi23}).

    The scheme is based on a space-time discretization defined by a final time $T>0$ and a pair $\Delta\coloneqq(\Delta x,\Delta t)$ such that
    \begin{equation*}
        0<\Delta x<|\gamma|,\quad \text{for every }\gamma\in\Ecal,\qquad \text{and}\qquad0<\Delta t<T.
    \end{equation*}
    To define the space-time grid, we fix an \emph{orientation} $\Ecal^+$, i.e., a subset of $\Ecal$ containing exactly one arc in each pair $\{\gamma,\wtgamma\}$, and set
    \begin{equation*}
        N^\Delta_\gamma\coloneqq\mleft\lceil\frac{|\gamma|}{\Delta x}\mright\rceil,\quad \text{for }\gamma\in\Ecal^+,\qquad \text{and}\qquad N^\Delta_T\coloneqq\mleft\lceil\frac T{\Delta t}\mright\rceil,
    \end{equation*}
    where $\lceil\cdot\rceil$ stands for the ceiling function. We define
    \begin{equation*}
        s^{\Delta,\gamma}_i\coloneqq i\frac{|\gamma|}{N^\Delta_\gamma},\quad \text{for }i\in\mleft\{0,\dotsc,N_\gamma^\Delta\mright\},\,\gamma\in\Ecal^+,\qquad t^\Delta_m\coloneqq m\frac T{N^\Delta_T},\quad \text{for }m\in\mleft\{0,\dotsc,N_T^\Delta\mright\}.
    \end{equation*}
    We introduce the grids associated to any $\gamma \in \Ecal^+$ and the grid on the interval $[0,T]$
    \begin{equation*}
        \Sscr_{\Delta,\gamma}\coloneqq\mleft\{s_i^{\Delta,\gamma}:i\in\mleft\{0,\dotsc,N_\gamma^\Delta\mright\}\mright\},\quad \text{for }\gamma\in\Ecal^+,\qquad\Tscr_\Delta^T\coloneqq\mleft\{t^\Delta_m:m\in\mleft\{0,\dotsc,N^\Delta_T\mright\}\mright\},
    \end{equation*}
    and finally the grid on $\Gamma$ and the grid on $\Gamma \times [0,T]$
    \begin{equation*}
        \Gamma^0_\Delta\coloneqq\bigcup_{\gamma\in\Ecal^+}\gamma(\Sscr_{\Delta,\gamma}),\qquad\Gamma_\Delta^T\coloneqq\Gamma_\Delta^0\times\Tscr^T_\Delta.
    \end{equation*}
    We further set $\Delta_\gamma x\coloneqq\dfrac{|\gamma|}{N_\gamma^\Delta}$ and say that $\Delta\coloneqq(\Delta x,\Delta t)$ is an \emph{admissible pair} whenever
    \begin{equation}\label{eq:goodpair}
        0<\Delta x<|\gamma|,\quad \text{for every }\gamma\in\Ecal,\qquad0<\Delta t<T\qquad \text{and}\qquad\Delta t\le\min_{\gamma\in\Ecal^+}\frac{\Delta_\gamma x}{\beta_0}.
    \end{equation}
    We denote by $B\mleft(\Gamma^0_\Delta\mright)$ and $B(\Sscr_{_\Delta,\gamma})$, for $\gamma\in\Ecal^+$, the spaces of bounded functions from $\Gamma^0_\Delta $ and $\Sscr_{_\Delta,\gamma}$, respectively, to $\Rds$. Given an arc $\gamma\in\Ecal^+$ and $f\in B(\Sscr_{_\Delta,\gamma})$, we define the operator
    \begin{equation*}
        S_\gamma[f](s)\coloneqq\min_{\frac{s-|\gamma|}{\Delta t}\le\lambda\le\frac s{\Delta t}}\{I_\gamma[f](s-\Delta t\lambda)+\Delta t\widetilde L_\gamma(s,\lambda)\},\qquad \text{for }s\in\Sscr_{_\Delta,\gamma},
    \end{equation*}
    where $I_\gamma$ is the \emph{interpolating polynomial} of degree 1 defined by
    \begin{equation*}
        I_\gamma[f](s)\coloneqq f(s_i)+\frac{s-s_i}{s_{i+1}-s_i}(f(s_{i+1})-f(s_i)),
    \end{equation*}
    with $s_i,s_{i+1}\in\Sscr_{_\Delta,\gamma}$ and $s\in[s_i,s_{i+1}]$. To define the scheme, we extend $S_\gamma$ to $B\mleft(\Gamma^0_\Delta \mright)$ setting the map $S:B\mleft(\Gamma^0_\Delta \mright)\to B\mleft(\Gamma^0_\Delta\mright)$ such that
    \begin{mylist}
        \item if $x\in\Gamma^0_\Delta \setminus\Vbf$ let $\gamma\in\Ecal^+$ and $s\in\Sscr_{_\Delta,\gamma}$ be such that $x=\gamma(s)$, then
        \begin{equation*}
            S[f](x)\coloneqq S_\gamma[f\circ\gamma](s);
        \end{equation*}
        \item if $x\in\Vbf$ we set $S$ through a two steps procedure:
        \begin{align*}
            \wtS[f](x)\coloneqq\,&\min\mleft\{S_\gamma[f]\mleft(\gamma^{-1}(x)\mright): \text{$\gamma\in\Ecal^+$ with $\gamma(0)=x$ or $\gamma(|\gamma|)=x$}\mright\},\\
            S[f](x)\coloneqq\,&\min\mleft\{\wtS[f](x),f(x)-c_x\Delta t\mright\}.
        \end{align*}
    \end{mylist}
    Given a Lipschitz initial datum $\phi$, the scheme computes the approximate solution $v^\Delta: \Gamma^T_\Delta \to \Rds$ to the time-dependent problem in the following way:
    \begin{equation}\label{eq:SLscheme}\tag{Discr}
        \mleft\{
        \begin{aligned}
            v^\Delta(x,0)=\,&\phi (x),&& \text{if }x\in\Gamma^0_\Delta ,\\
            v^\Delta(x,t)=\,&S\mleft[v^\Delta(\cdot,t-\Delta t)\mright](x),&& \text{if }t\ne0,\,(x,t)\in\Gamma^T_\Delta.
        \end{aligned}
        \mright.
    \end{equation}
    The following error estimate holds for this numerical scheme:

    \begin{theo}\label{errCS}
        \emph{\cite[Corollary 4.14]{CarliniCoscettiPozza24}} Fixed a Lipschitz continuous initial datum $\phi$, $T>0$ and an admissible pair $\Delta$, let $v^\Delta$ be the approximation of $\Scal(t)\phi$ on $\Gamma^T_\Delta$ given by the semi-Lagrangian scheme~\eqref{eq:SLscheme}. There exists a constant $C>0$, independent of $\Delta$ and $T$, such that
        \begin{equation*}
            \mleft|v^\Delta(x,t)-(\Scal(t)\phi)(x)\mright|\le CT\mleft(\frac1{\sqrt{\Delta t}}\min\mleft\{\Delta x,\frac{\Delta x^2}{\Delta t}\mright\}+\sqrt{\Delta t}\mright),\qquad \text{for any }(x,t)\in\Gamma^T_\Delta.
        \end{equation*}
    \end{theo}

    \begin{rem}
        We stress out that, in principle, our algorithms could employ any other numerical method for the approximation of solutions to~\eqref{eq:globteik}. In view of~\cite[Theorem 3.4]{CarliniCoscettiPozza24}, which proves that our definition of solution is equivalent to the one proposed in~\cite{ImbertMonneau17}, we could even choose the finite difference schemes proposed in~\cite{CamilliCarliniMarchi18,CostesequeLebacqueMonneau14} or the semi-Lagrangian method outlined in~\cite{CarliniFestaForcadel20}.
    \end{rem}

    \subsection{An Algorithm with an A Priori Error Estimate}

    Using an approach employed in~\cite{AchdouCamilliCapuzzo-Dolcetta08} for equations posed on the $N$-dimensional torus, we obtain an algorithm complete with an error estimate which proves the convergence of the scheme. The main idea behind this algorithm is to exploit the estimates given in \cref{basictheo,basicbounds}.

    \Cref{basicalgo} consists of two moduli: $\Mcal_1$ and $\Mcal_2$. $\Mcal_1$ is the scheme~\eqref{eq:SLscheme}, takes as inputs a Lipschitz continuous initial datum $\phi$, a final time $T>0$ and an admissible pair $\Delta\coloneqq(\Delta x,\Delta t)$, and gives as output an approximation $v^\Delta$ of the exact solution to~\eqref{eq:globteik} with initial datum $\phi$, presented at the start of this \lcnamecref{algosec}, at time $T$:
    \begin{equation*}
        \Mcal_1:(\phi,T,\Delta)\longmapsto v^\Delta(\cdot,T).
    \end{equation*}
    $\Mcal_2$ is a control modulus on the outputs of $\Mcal_1$. For $f$, $g$ defined on the grid $\Gamma_\Delta^0$ and $A>0$,
    \begin{equation*}
        \Mcal_2:(f,g,A,\Delta x)\longmapsto\mleft(\min_{x\in\Gamma_\Delta^0}\frac{f(x)-g(x)}A,\max_{x\in\Gamma_\Delta^0}\frac{f(x)-g(x)}A\mright).
    \end{equation*}
    Fixed a Lipschitz continuous initial value $\phi$, a time $T>0$ and an admissible pair $\Delta$, these modules are used to define, at each iteration of the algorithm,
    \begin{equation*}
        \ovc^\Delta_k\coloneqq\max_{x\in\Gamma_\Delta^0}\frac{\phi(x)-v^\Delta(x,kT)}{kT}\qquad \text{and}\qquad\unc^\Delta_k\coloneqq\min_{x\in\Gamma_\Delta^0}\frac{\phi(x)-v^\Delta(x,kT)}{kT},
    \end{equation*}
    where $k$ is the number of iterations. The algorithm stops when $\mleft(\ovc^\Delta_k-\unc^\Delta_k\mright)/2$ is smaller than a given tolerance and returns $\mleft(\ovc^\Delta_k+\unc^\Delta_k\mright)/2$ as approximation of the critical value $c$. We point out that $\Delta t$ defines the time step used internally by $\Mcal_1$, while $T$ defines the time step used by our method at each step.

    \begin{algorithm}
        \caption{Algorithm for the critical value with an a priori error estimate}\label{basicalgo}
        \begin{algorithmic}[1]
            \Require{a Lipschitz continuous initial value $\phi$, a time $T>0$, an admissible pair $\Delta$ and a tolerance $\varepsilon>0$}
            \Procedure{Critical Value}{$\phi,T,\Delta,\varepsilon$}
                \State{$k\gets0$}
                \Repeat
                    \State{$k\gets k+1$}
                    \State{$v^\Delta(\cdot,kT)\gets\Mcal_1(\phi,kT,\Delta)$}
                    \State{use $\Mcal_2$ to get
                        \begin{equation*}
                            \ovc^\Delta_k\gets\max_{x\in\Gamma_\Delta^0}\frac{\phi(x)-v^\Delta(x,kT)}{kT}\quad \text{and}\quad\unc^\Delta_k\gets\min_{x\in\Gamma_\Delta^0}\frac{\phi(x)-v^\Delta(x,kT)}{kT}
                        \end{equation*}
                    }
                \Until{$\ovc^\Delta_k-\unc^\Delta_k<2\varepsilon$}\label{alg:basicalgo1}
                \State{\Return $\dfrac{\ovc^\Delta_k+\unc^\Delta_k}2$}
            \EndProcedure
        \end{algorithmic}
    \end{algorithm}

    \begin{rem}\label{basicimprov}
        In view of \cref{basicbounds} and~\eqref{eq:critvaldef}, \cref{basicalgo} can be improved by taking
        \begin{equation*}
            \ovc^\Delta_1\coloneqq\max_{x\in\Gamma_\Delta^0}\frac{\phi(x)-v^\Delta(x,T)}T,\qquad\unc^\Delta_1\coloneqq\max\mleft\{a_0,\min_{x\in\Gamma_\Delta^0}\frac{\phi(x)-v^\Delta(x,T)}T\mright\}
        \end{equation*}
        and, for $k>1$,
        \begin{equation*}
            \ovc^\Delta_k\coloneqq\min\mleft\{\ovc^\Delta_{k-1},\max_{x\in\Gamma_\Delta^0}\frac{\phi(x)-v^\Delta(x,kT)}{kT}\mright\},\qquad\unc^\Delta_k\coloneqq\max\mleft\{\unc^\Delta_{k-1},\min_{x\in\Gamma_\Delta^0}\frac{\phi(x)-v^\Delta(x,kT)}{kT}\mright\}.
        \end{equation*}
    \end{rem}

    \begin{theo}\label{basicerr}
        Given a Lipschitz continuous initial datum $\phi$, a $T>0$ and an admissible pair $\Delta$, let $v^\Delta(\cdot,kT)\coloneqq\Mcal_1(\phi,kT,\Delta)$ for $k\in\Nds$. We recursively set, for $k\ge1$
        \begin{equation*}
            \ovc^\Delta_k\coloneqq\max_{x\in\Gamma_\Delta^0}\frac{\phi(x)-v^\Delta(x,kT)}{kT},\qquad\unc^\Delta_k\coloneqq\min_{x\in\Gamma_\Delta^0}\frac{\phi(x)-v^\Delta(x,kT)}{kT}.
        \end{equation*}
        We have that
        \begin{equation*}
            \mleft|\frac{\ovc^\Delta_k+\unc^\Delta_k}2-c\mright|\le C_0\mleft(\frac{1+\Delta x}{kT}+\frac1{\sqrt{\Delta t}}\min\mleft\{\Delta x,\frac{\Delta x^2}{\Delta t}\mright\}+\sqrt{\Delta t}\mright),\qquad \text{for any }k\ge1,
        \end{equation*}
        where $C_0>0$ is a constant independent of $\Delta$ and $k$.
    \end{theo}

    \Cref{basicerr} proves that the approximation given by \cref{basicalgo} converges to the critical value as $k$ diverges and, if~\eqref{eq:goodpair} holds,
    \begin{equation*}
        \min\mleft\{\frac{\Delta x}{\sqrt{\Delta t}},\frac{\Delta x^2}{\Delta t^{\frac32}}\mright\}\longrightarrow0.
    \end{equation*}

    \begin{proof}[Proof of \cref{basicerr}]
        Fixed $k\ge1$, we set
        \begin{equation*}
            \ovx\coloneqq\amax_{x\in\Gamma}\phi(x)-(\Scal(kT)\phi)(x),\qquad\ovy\coloneqq\amin_{y\in\Gamma_\Delta^0}|y-\ovx|.
        \end{equation*}
        Clearly $|\ovx-\ovy|\le\dfrac{\Delta x}2$, thereby
        \begin{equation}\label{eq:basicerr1}
            \begin{aligned}
                0\le\,&\max_{x\in\Gamma}(\phi(x)-(\Scal(kT)\phi)(x))-\max_{y\in\Gamma_\Delta^0}(\phi(y)-(\Scal(kT)\phi)(y))\\
                \le\,&\phi(\ovx)-(\Scal(kT)\phi)(\ovx)-\phi(\ovy)+(\Scal(kT)\phi)(\ovy)\le\ell\Delta x,
            \end{aligned}
        \end{equation}
        where $\ell$ is the Lipschitz constant of $(x,t)\mapsto(\Scal(t)\phi)(x)$. Then, by \cref{errCS}, there is $C>0$ independent of $\Delta$ and $k$ such that
        \begin{align*}
            \mleft|\ovc_k^\Delta-\max_{x\in\Gamma}\frac{\phi(x)-(\Scal(kT)\phi)(x)}{kT}\mright|\le\,&\frac{\ell\Delta x}{kT}+C\mleft(\frac1{\sqrt{\Delta t}}\min\mleft\{\Delta x,\frac{\Delta x^2}{\Delta t}\mright\}+\sqrt{\Delta t}\mright),\\
            \mleft|\unc_k^\Delta-\min_{x\in\Gamma}\frac{\phi(x)-(\Scal(kT)\phi)(x)}{kT}\mright|\le\,&\frac{\ell\Delta x}{kT}+C\mleft(\frac1{\sqrt{\Delta t}}\min\mleft\{\Delta x,\frac{\Delta x^2}{\Delta t}\mright\}+\sqrt{\Delta t}\mright),
        \end{align*}
        thus \cref{basictheo} yields
        \begin{align*}
            \mleft|\frac{\ovc^\Delta_k+\unc^\Delta_k}2-c\mright|\le&\mleft|\max_{x\in\Gamma}\frac{\phi(x)-(\Scal(kT)\phi)(x)}{2kT}+\min_{x\in\Gamma}\frac{\phi(x)-(\Scal(kT)\phi)(x)}{2kT}-c\mright|\\
            &+\frac{\ell\Delta x}{kT}+C\mleft(\frac1{\sqrt{\Delta t}}\min\mleft\{\Delta x,\frac{\Delta x^2}{\Delta t}\mright\}+\sqrt{\Delta t}\mright)\\
            \le&\,2\ell\diam(\Gamma)\frac1{kT}+\frac{\ell\Delta x}{kT}+C\mleft(\frac1{\sqrt{\Delta t}}\min\mleft\{\Delta x,\frac{\Delta x^2}{\Delta t}\mright\}+\sqrt{\Delta t}\mright).
        \end{align*}
        Since $k$ is arbitrary this concludes the proof.
    \end{proof}

    Finally, we know from \cref{basicbounds} that, for any $k\ge1$,
    \begin{multline*}
        \mleft|\max_{x\in\Gamma}\frac{\phi(x)-(\Scal(kT)\phi)(x)}{2kT}+\min_{x\in\Gamma}\frac{\phi(x)-(\Scal(kT)\phi)(x)}{2kT}-c\mright|\\
        \le\mleft|\max_{x\in\Gamma}\frac{\phi(x)-(\Scal(kT)\phi)(x)}{2kT}-\min_{x\in\Gamma}\frac{\phi(x)-(\Scal(kT)\phi)(x)}{2kT}\mright|,
    \end{multline*}
    thus, arguing as in the proof of \cref{basicerr}, we get the following \namecref{basicstop} which justifies the stopping criterion $\ovc^\Delta_k-\unc^\Delta_k<2\varepsilon$ at \cref{alg:basicalgo1} in \cref{basicalgo}.

    \begin{prop}\label{basicstop}
        Under the same assumptions of \cref{basicerr}, we have that
        \begin{equation*}
            \mleft|\frac{\ovc^\Delta_k+\unc^\Delta_k}2-c\mright|\le\frac{\ovc_k^\Delta-\unc_k^\Delta}2+C_0\mleft(\frac{\Delta x}{kT}+\frac1{\sqrt{\Delta t}}\min\mleft\{\Delta x,\frac{\Delta x^2}{\Delta t}\mright\}+\sqrt{\Delta t}\mright),\qquad \text{for any }k\ge1,
        \end{equation*}
        where $C_0>0$ is a constant independent of $\Delta$ and $k$.
    \end{prop}

    \subsection{A Faster Iterative Algorithm}

    Next we describe a numerical scheme which exploit \cref{itertheo}, instead of \cref{basictheo}, to obtain an approximation of the critical value. While we do not have an a priori error estimate for this new algorithm, the simulations in \cref{simsec} show that it is considerably faster than \cref{basicalgo}.

    Following \cref{iterapprox}, $c_k$ defined in~\eqref{eq:iterapprox.1} is a reasonable approximation of $c$ for $k$ big enough, thus we propose a scheme, \cref{iteralgo}, which estimates the critical value using a numerical approximation of $c_k$. It consists of the same two moduli $\Mcal_1$ and $\Mcal_2$ used in \cref{basicalgo}. They are used to defining at each iteration, fixed a Lipschitz continuous initial datum $\phi$, a time $T>0$ and an admissible pair $\Delta$,
    \begin{equation*}
        \wtc^\Delta_k\coloneqq\max_{x\in\Gamma_\Delta^0}\frac{v^\Delta(x,(k-1)T)-v^\Delta(x,kT)}T\qquad \text{and}\qquad\utc^\Delta_k\coloneqq\min_{x\in\Gamma_\Delta^0}\frac{v^\Delta(x,(k-1)T)-v^\Delta(x,kT)}T,
    \end{equation*}
    where $k$ is the number of iterations. The algorithm stops when $\mleft(\wtc^\Delta_k-\utc^\Delta_k\mright)/2$ is smaller than a given tolerance and returns $\mleft(\wtc^\Delta_k+\utc^\Delta_k\mright)/2$ as approximation of the critical value $c$. We recall that $\Delta t$ defines the time step used internally by $\Mcal_1$, while $T$ defines the time step used by our method at each step.

    \begin{algorithm}
        \caption{A faster iterative algorithm for the critical value}\label{iteralgo}
        \begin{algorithmic}[1]
            \Require{a Lipschitz continuous initial value $\phi$, a time $T>0$, an admissible pair $\Delta$ and a tolerance $\varepsilon>0$}
            \Procedure{Critical Value}{$\phi,T,\Delta,\varepsilon$}
                \State{$k\gets0$}
                \Repeat
                    \State{$k\gets k+1$}
                    \State{$v^\Delta(\cdot,kT)\gets\Mcal_1(\phi,kT,\Delta)$}
                    \State{use $\Mcal_2$ to get
                        \begin{equation*}
                            \wtc^\Delta_k\gets\max_{x\in\Gamma_\Delta^0}\frac{v^\Delta(x,(k-1)T)-v^\Delta(x,kT)}T\quad \text{and}\quad\utc^\Delta_k\gets\min_{x\in\Gamma_\Delta^0}\frac{v^\Delta(x,(k-1)T)-v^\Delta(x,kT)}T
                        \end{equation*}
                    }
                \Until{$\wtc^\Delta_k-\utc^\Delta_k<2\varepsilon$}\label{alg:iteralgo1}
                \State{\Return $\dfrac{\wtc^\Delta_k+\utc^\Delta_k}2$}
            \EndProcedure
        \end{algorithmic}
    \end{algorithm}

    \begin{rem}\label{iterimprov}
        In view of \cref{itertheo} and~\eqref{eq:critvaldef}, \cref{iteralgo} can be improved by taking
        \begin{equation*}
            \wtc^\Delta_1\coloneqq\max_{x\in\Gamma_\Delta^0}\frac{\phi(x)-v^\Delta(x,T)}T,\qquad \utc^\Delta_1\coloneqq\max\mleft\{a_0,\min_{x\in\Gamma_\Delta^0}\frac{\phi(x)-v^\Delta(x,T)}T\mright\}
        \end{equation*}
        and, for $k>1$,
        \begin{align*}
            \wtc^\Delta_k&\coloneqq\min\mleft\{\wtc^\Delta_{k-1},\max_{x\in\Gamma_\Delta^0}\frac{v^\Delta(x,(k-1)T)-v^\Delta(x,kT)}{T}\mright\},\\
            \utc^\Delta_k&\coloneqq\max\mleft\{\utc^\Delta_{k-1},\min_{x\in\Gamma_\Delta^0}\frac{v^\Delta(x,(k-1)T)-v^\Delta(x,kT)}{T}\mright\}.
        \end{align*}
    \end{rem}

    The next \namecref{iterconv} shows the convergence of \cref{iteralgo}.

    \begin{theo}\label{iterconv}
        Given a Lipschitz continuous initial datum $\phi$, a $T>0$ and an admissible pair $\Delta$, let $v^\Delta(\cdot,kT)\coloneqq\Mcal_1(\phi,kT,\Delta)$ for $k\in\Nds$. Setting for $k\ge1$
        \begin{equation*}
            \wtc^\Delta_k\coloneqq\max_{x\in\Gamma_\Delta^0}\frac{v^\Delta(x,(k-1)T)-v^\Delta(x,kT)}T,\qquad\utc^\Delta_k\coloneqq\min_{x\in\Gamma_\Delta^0}\frac{v^\Delta(x,(k-1)T)-v^\Delta(x,kT)}T
        \end{equation*}
        and $c_k$ as in \cref{iterapprox}, we have that,
        \begin{equation}\label{eq:iterconv.1}
            \mleft|\frac{\wtc^\Delta_k+\utc^\Delta_k}2-c\mright| \to 0
        \end{equation}
        as $k\to \infty$ and $\Delta \to 0$ such that
        \begin{equation}\label{eq:iterconvexp}
            k\min\mleft\{\frac{\Delta x}{\sqrt{\Delta t}},\frac{\Delta x^2}{\Delta t^{\frac32}}\mright\}\longrightarrow0.
        \end{equation}
    \end{theo}
    \begin{proof}
        We start fixing $k\ge1$. The same arguments used to prove~\eqref{eq:basicerr1} show that
        \begin{equation*}
            0\le\max_{x\in\Gamma}((\Scal((k-1)T)\phi)(x)-(\Scal(kT)\phi)(x))-\max_{y\in\Gamma_\Delta^0}((\Scal((k-1)T)\phi)(y)-(\Scal(kT)\phi)(y))\le\ell\Delta x,
        \end{equation*}
        where $\ell$ is the Lipschitz constant of $(x,t)\mapsto(\Scal(t)\phi)(x)$. Setting $\wtc_k$ and $\utc_k$ as in \cref{itertheo}, \cref{errCS} yields that there is a $C>0$ independent of $\Delta$ and $k$ such that
        \begin{equation*}
            \mleft|\frac{\wtc^\Delta_k+\utc^\Delta_k}2-\frac{\wtc_k+\utc_k}2\mright|\le\frac{\ell\Delta x}T+C(2k+1)\mleft(\frac1{\sqrt{\Delta t}}\min\mleft\{\Delta x,\frac{\Delta x^2}{\Delta t}\mright\}+\sqrt{\Delta t}\mright),
        \end{equation*}
        therefore
        \begin{equation*}
            \mleft|\frac{\wtc^\Delta_k+\utc^\Delta_k}2-c\mright|\le|c_k-c|+\frac{\ell\Delta x}T+C(2k+1)\mleft(\frac1{\sqrt{\Delta t}}\min\mleft\{\Delta x,\frac{\Delta x^2}{\Delta t}\mright\}+\sqrt{\Delta t}\mright).
        \end{equation*}
        This and \cref{iterapprox} imply~\eqref{eq:iterconv.1} whenever~\eqref{eq:goodpair} and~\eqref{eq:iterconvexp} hold.
    \end{proof}

    Adopting the same notation used in the proof of \cref{iterconv}, we notice that \cref{iterapprox} yields
    \begin{equation*}
        |c_k-c|\le\frac{\wtc_k-\utc_k}2\le\frac{\wtc^\Delta_k-\utc^\Delta_k}2+\frac{\ell\Delta x}T+C(2k+1)\mleft(\frac1{\sqrt{\Delta t}}\min\mleft\{\Delta x,\frac{\Delta x^2}{\Delta t}\mright\}+\sqrt{\Delta t}\mright).
    \end{equation*}
    This and \cref{iterconv} prove the following error estimate for \cref{iteralgo}.

    \begin{prop}\label{itererr}
        Under the same assumptions of \cref{iterconv}, we have that, for any $k\ge1$,
        \begin{equation}\label{eq:itererr.1}
            \mleft|\frac{\wtc^\Delta_k+\utc^\Delta_k}2-c\mright|\le\frac{\wtc^\Delta_k-\utc^\Delta_k}2+C_0\mleft(\frac{\Delta x}T+(k+1)\mleft(\frac1{\sqrt{\Delta t}}\min\mleft\{\Delta x,\frac{\Delta x^2}{\Delta t}\mright\}+\sqrt{\Delta t}\mright)\mright),
        \end{equation}
        where $C_0>0$ is a constant independent of $\Delta$ and $k$.
    \end{prop}

    \Cref{itererr} justifies the stopping criterion $\wtc^\Delta_k-\utc^\Delta_k<2\varepsilon$ at \cref{alg:iteralgo1} in \cref{iteralgo}, but does not prove the convergence of the algorithm, as \cref{iterconv}, due to the presence of the terms $\wtc^\Delta_k$ and $\utc_k^\Delta$.

    \begin{rem}
        We point out that the presence of $k$ in the numerator of both~\eqref{eq:iterconvexp} and~\eqref{eq:itererr.1} affects the estimated convergence rate of our method. It also means that the approximation provided by~\cref{iteralgo} becomes less reliable with each step, in contrast with \cref{basicalgo}. Fixed an admissible pair $\Delta$, one could exploit~\eqref{eq:itererr.1} to pick a maximal number of iterations after which the algorithm should not be trusted, namely, fix a maximal number of iterations \cref{iteralgo} can perform reliably.
    \end{rem}

    \section{Numerical Simulations}\label{simsec}

    In this section we present numerical simulations performed on two types of networks in $\Rds^2$ to show the performance and convergence properties of the proposed algorithms. In the first case we consider a simple network in the form of an equilateral triangle. In the second one we investigate a more complex network structured as a traffic circle. In both settings we discuss one problem with a Hamiltonian dependent on $s$ and one problem with a Hamiltonian independent of $s$.

    We now focus on the selection of parameters required by \cref{basicalgo}. Based on the theoretical results established in the previous sections, we examine the expected performance of the algorithm as a function of these chosen parameters. First, we select a sequence of pairs $\Delta $ such that
    \begin{equation}\label{eq:notenoughgoodpair}
        \min\mleft\{\frac{\Delta x}{\sqrt{\Delta t}},\frac{\Delta x^2}{\Delta t^{\frac32}}\mright\}\longrightarrow0 \qquad \text{as }\Delta \to 0.
    \end{equation}
    If $\Delta t$ is further chosen to satisfy~\eqref{eq:goodpair}, then \cref{basicerr} ensures that for $k$ large enough the error $\mleft|\frac{\ovc^\Delta_k + \unc^\Delta_k}{2} - c\mright|$ is bounded above by a term proportional to $\min\mleft\{\frac{\Delta x}{\sqrt{\Delta t}}, \frac{\Delta x^2}{\Delta t^{3/2}}\mright\}$ and independent of $k$. This term arises solely from the error estimate for the semi-Lagrangian scheme employed to approximate the solution to~\eqref{eq:globteik} (as shown in \cref{errCS}), it remains constant for fixed $\Delta$ and vanishes as $\Delta \to 0$. Equivalently, by selecting a small enough stopping condition $\varepsilon$ in \cref{basicalgo}, \cref{basicstop} implies that the error $\mleft|\frac{\ovc^\Delta_k + \unc^\Delta_k}{2} - c\mright|$ is bounded above by a term proportional to $\min\mleft\{\frac{\Delta x}{\sqrt{\Delta t}}, \frac{\Delta x^2}{\Delta t^{3/2}}\mright\}$ when the algorithm reaches this tolerance. As before, this term remains constant for fixed $\Delta$ and tends to zero as $\Delta \to 0$. Consequently, this analysis enables us to conclude that, for a fixed $\Delta$, the tolerance $\varepsilon$ can be chosen as a function of $\Delta$ without any loss of accuracy.

    In the next simulations we implement the algorithm with time step $\Delta t= (\min_{\gamma\in\Ecal^+} \Delta x_\gamma)/\beta_0$, for which both~\eqref{eq:notenoughgoodpair} and~\eqref{eq:goodpair} hold. We also evaluate the algorithm for $\Delta t= \Delta x/2$ and $\Delta t= \Delta x^{5/6}$, which satisfy~\eqref{eq:notenoughgoodpair} but do not satisfy~\eqref{eq:goodpair}. These last two choices do not meet the assumptions in \cref{errCS} and, consequentially, in \cref{basicerr}, but allow us for larger time steps. However, the last choice aligns with the convergence properties of the semi-Lagrangian scheme for the time-dependent Hamilton--Jacobi equation~\eqref{eq:globteik} established in~\cite{CarliniSiconolfi23}.

    We additionally perform these tests for \cref{iteralgo} to compare its performance with \cref{basicalgo}. We point out that, while \cref{basicerr} guarantees convergence of \cref{basicalgo} as $\Delta \to 0$ suitably and as $k\to \infty$ independently of $\Delta $, the convergence of \cref{iteralgo} is proven if~\eqref{eq:iterconvexp} holds. Moreover, while \cref{basicstop} ensures that, by choosing a sufficiently small stopping condition, the difference between the approximated and exact critical value is bounded above by a term dependent on $\Delta $ and independent of $k$ when \cref{basicalgo} reaches the prescribed tolerance, \cref{itererr} does not provide a similar result for \cref{iteralgo}. However, from a numerical point of view, \cref{iteralgo} turns out to be highly advantageous. Indeed, the following simulations show that it attains the same accuracy as \cref{basicalgo} while reducing the number of iterations by about 81\% on average.

    We implement \cref{basicalgo,iteralgo} with the improvements detailed in \cref{basicimprov,iterimprov}. With a slight abuse of notation, we denote by $c^\Delta$ the output returned by both \cref{basicalgo} and \cref{iteralgo}. The accuracy is measured by computing the absolute value of the difference between the approximated critical value $c^{\Delta}$ and a reference critical value $\hat c$ or the exact critical value $c$, when it is known. We apply the algorithms in two ways: either by fixing the number of iterations to be executed or by setting a tolerance upon reaching which the algorithm terminates. In the following tables, the column labeled $k$ reports the number of iterations performed, either to complete the predetermined number of iterations or to satisfy the prescribed tolerance. In the following simulations, we run the algorithms with $\phi \equiv 0$, $T=1$ and minimal flux limiters.

    \subsection{Test on a Triangle}

    We consider as $\Gamma \subset \Rds^2$ the triangle with vertices
    \begin{equation*}
        z_1 \coloneqq(0, 0), \qquad z_2\coloneqq\mleft(\frac{1}{2},\frac{\sqrt{3}}{2}\mright), \qquad z_3\coloneqq(1,0),
    \end{equation*}
    and arcs $\gamma_1,\gamma_2,\gamma_3:[0,1]\longrightarrow \Rds^2$
    \begin{equation*}
        \gamma_1(s) \coloneqq sz_2, \qquad \gamma_2(s) \coloneqq(1-s)z_2+sz_3,\qquad \gamma_3 (s)\coloneqq(1-s)z_3,
    \end{equation*}
    plus the reversed arcs (\cref{Network:triangle}).

    \begin{figure}[!ht]
        \centering
        \begin{tikzpicture}[scale=3]
            \draw[->] (-0.5,0) -- (1.5,0) node[right] {$x$};
            \draw[->] (0,-0.5) -- (0,1.5) node[above] {$y$};

            \node[below left] at (0,0) {$z_1$};
            \node[below right] at (1,0) {$z_3$};
            \node[above] at (0.5,{sqrt(3)/2}) {$z_2$};

            \draw[thick, postaction={decorate, decoration={markings, mark=at position 0.5 with {\arrow{>}}}}] (0,0) -- (0.5,{sqrt(3)/2})
            node[midway, left] {$\gamma_1$};

            \draw[thick, postaction={decorate, decoration={markings, mark=at position 0.5 with {\arrow{>}}}}] (0.5,{sqrt(3)/2}) -- (1,0)
            node[midway, right] {$\gamma_2$};

            \draw[thick, postaction={decorate, decoration={markings, mark=at position 0.5 with {\arrow{>}}}}] (1,0) -- (0,0)
            node[midway, below] {$\gamma_3$};

            \filldraw[black] (0,0) circle (0.5pt);
            \filldraw[black] (1,0) circle (0.5pt);
            \filldraw[black] (0.5,{sqrt(3)/2}) circle (0.5pt);
        \end{tikzpicture}
        \caption{Network $\Gamma \subset \Rds^2 $: the triangle.}\label{Network:triangle}
    \end{figure}
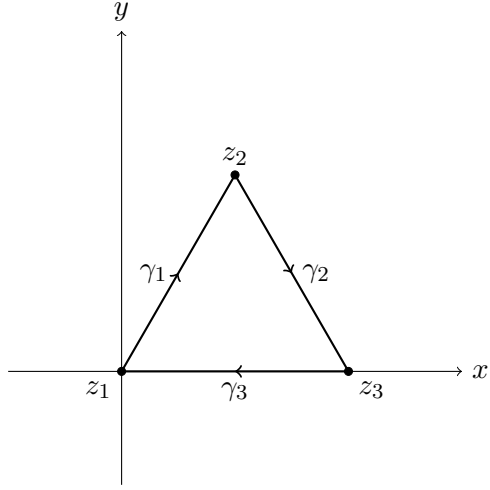

    \subsubsection{Hamiltonian Dependent on \texorpdfstring{$s$}{s}}\label{section:TriDependent}

    Let us take the Hamiltonians $H_{\gamma_i}: [0,1]\times \Rds\longrightarrow \Rds$ for $i=1,2,3$
    \begin{equation*}
        H_{\gamma_1}(s,\mu)\coloneqq(\mu+2s)^2,\qquad H_{\gamma_2}(s,\mu)\coloneqq\mu^2+2Bs,\qquad H_{\gamma_3}(s,\mu)\coloneqq(\mu+A-1+2s)(\mu+2A)+C,
    \end{equation*}
    where $A$, $B$, $C$ are nonnegative constants such that $C\ge2B$. It is easy to prove that if
    \begin{equation*}
        A=\sqrt C-1+\frac2{6B}\mleft(C^{\frac32}-(C-2B)^{\frac32}\mright),
    \end{equation*}
    then $c=C$ is the exact critical value of the problem. We study the example with $A=2/3$, $B=1/2$ and $C=1$. For this problem we can choose $\beta_0=12$.

    We begin evaluating the performance of \cref{basicalgo} in relation to the theoretical results explored at the beginning of \cref{simsec}. This analysis further supports the choice of the tolerance selected in the following tests. Specifically, we fix $2000$ iterations to be executed, set the time step $\Delta t = (\min_{\gamma\in\Ecal^+} \Delta x_\gamma)/12$ and implement the algorithm for several refinements of the space grid. The results are illustrated in \cref{fig:TriDep} and summarized in \cref{table:TriangleDepAlgo1k2000}. In \cref{fig:TriDep} (left), the behavior of $(\ovc^\Delta_k - \unc^\Delta_k)/2$ is displayed for several values of the spatial step $\Delta x$. As expected, this quantity decreases as $k$ increases. Consequently, setting a small tolerance in the stopping condition requires the algorithm to perform a large number of iterations. In \cref{fig:TriDep} (right), the error $\mleft| \frac{\ovc^\Delta_k + \unc^\Delta_k}{2} - c \mright|$ is shown for several values of $\Delta x$. Consistently with the discussion at the beginning of \cref{simsec}, this error stabilizes and reaches a constant value after a sufficient number of iterations. This result suggests that, for a fixed $\Delta x$, running all $2000$ iterations is unnecessary and the algorithm can be stopped once the plateau is reached. Examining the behavior as $\Delta x$ varies, as once again predicted by the previous theoretical analysis, it can be seen that the value of the aforementioned plateau decreases as $\Delta x$ decreases. Additionally, the number of iterations required to reach the plateau increases (or equivalently, the tolerance $\varepsilon$ required to stop the algorithm must be smaller) as $\Delta x$ decreases. These observations justify the following choice of tolerance proportional to $\Delta x$.

    \begin{figure}[!ht]
        \centering
        \subfigure
        {
            \includegraphics[width=6cm]{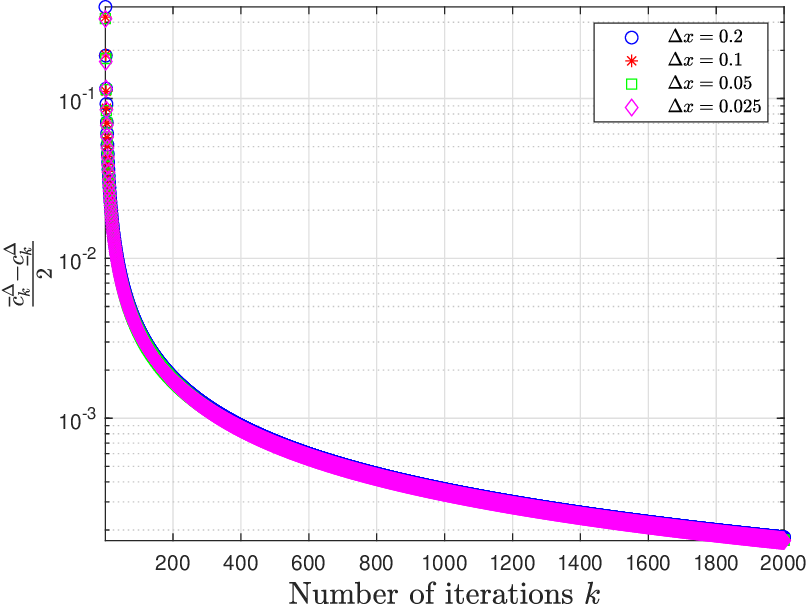}
        }
        \hspace{5mm}
        \subfigure
        {
            \includegraphics[width=6cm]{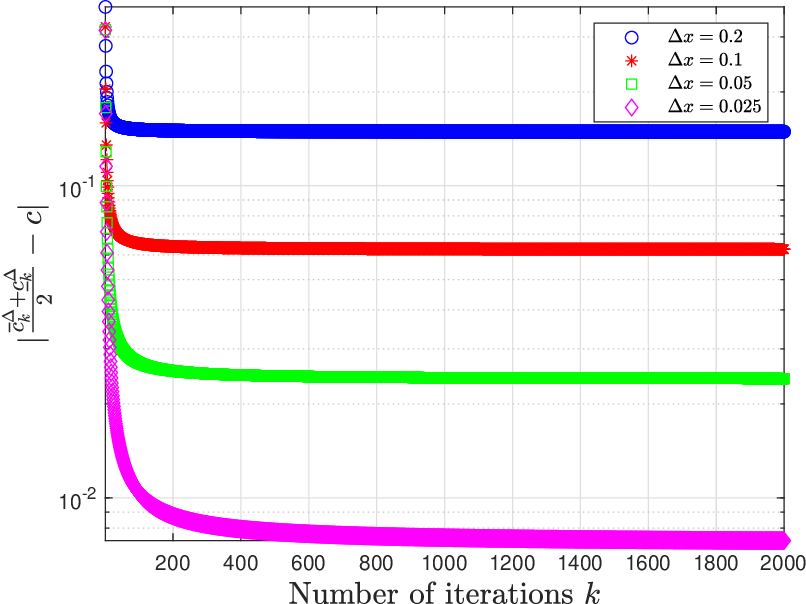}
        }
        \caption{\Cref{basicalgo}. Test on the triangle with Hamiltonian dependent on $s$, $A=2/3$, $B=1/2$, $C=1$, $\Delta t=\min_{\gamma\in\Ecal^+}\Delta_\gamma x/12$ and $2000$ fixed iterations.}\label{fig:TriDep}
    \end{figure}

    \begin{table}[!ht]
        \renewcommand\arraystretch{1.1}
        \centering
        \begin{tabular}{llll}
            \toprule
            ${\Delta x}$ & $k$ & $c^{\Delta}$ &$|c^{\Delta}-c |$ \\
            \midrule
            $2.00\cdot 10^{-1}$ & $2000$ & $1.15$ &$1.49\cdot 10^{-1}$ \\
            $1.00\cdot 10^{-1}$ & $2000$ & $1.06$ &$6.26\cdot 10^{-2} $ \\
            $5.00\cdot 10^{-2}$ & $2000$ & $1.02$ &$2.41\cdot 10^{-2} $ \\
            $2.50\cdot 10^{-2}$& $2000$ & $1.01$ &$7.29\cdot 10^{-3} $ \\
            \bottomrule
        \end{tabular}
        \caption{\Cref{basicalgo}. Test on the triangle with Hamiltonian dependent on $s$, $A=2/3$, $B=1/2$, $C=1$, $\Delta t=(\min_{\gamma\in\Ecal^+}\Delta_\gamma x)/12$ and $2000$ fixed iterations.}\label{table:TriangleDepAlgo1k2000}
    \end{table}

    We now implement \cref{basicalgo} with $\varepsilon = \Delta x/10$, $\Delta t = (\min_{\gamma \in \Ecal^+} \Delta x_\gamma) / 12$ and several refinements of the space grid. The results are presented in columns 8–10 of \cref{table:TriangleDepAlgo1}. Column 10 shows a decreasing trend in the error as $\Delta x$ decreases. As previously justified theoretically and observed in \cref{fig:TriDep} (right), this trend is independent of the choice of tolerance scaled by $\Delta x$. Indeed, comparing the errors in column 4 of \cref{table:TriangleDepAlgo1k2000} with those in column 10 of \cref{table:TriangleDepAlgo1}, the approximated critical value is computed with the same order of error. Additionally, we implement \cref{basicalgo} using the same tolerance $\varepsilon = \Delta x/10$ but varying the time step: $\Delta t = \Delta x^{5/6}$ (columns 2–4 of \cref{table:TriangleDepAlgo1}) and $\Delta t = \Delta x / 2$ (columns 5–7 of \cref{table:TriangleDepAlgo1}). \Cref{table:TriangleDepAlgo1} reveals that the worst approximation is obtained with $\Delta t = \Delta x^{5/6}$, while choosing $\Delta t = \Delta x / 2$ produces errors comparable to those achieved with $\Delta t = \min_{\gamma \in \Ecal^+} \Delta x_\gamma / 12$ preserving almost the same number of iterations. This choice offers the benefit of a larger time step, thereby reducing the computational cost.

    \begin{table}[!ht]
        \renewcommand\arraystretch{1.1}
        \centering
        \begin{tabular}{llllllllll}
            \toprule
            ${\Delta x}$ & $k$ & $c^{\Delta}$ &$|c^{\Delta}-c |$ & $k$ & $c^{\Delta}$ & $|c^{\Delta}-c |$ & $k$ & $c^{\Delta}$& $|c^{\Delta}-c |$ \\
            \midrule
            $2.00\cdot 10^{-1}$ & $40$ & $1.91$ &$9.11\cdot 10^{-1} $ & $19$& $1.20$ &$1.95\cdot 10^{-1} $ & $18$ & $1.16$ &$1.63\cdot 10^{-1} $ \\
            $1.00\cdot 10^{-1}$ & $46$ & $1.30$ &$2.99\cdot 10^{-1} $ & $36$ & $1.08$&$7.86\cdot 10^{-2} $ & $35$ & $1.07$&$7.07\cdot 10^{-2} $ \\
            $5.00\cdot 10^{-2}$ & $75$ & $1.10$ &$1.00\cdot 10^{-1} $ & $70$ & $1.03$&$3.02\cdot 10^{-2} $ & $69$ & $1.03$&$2.84\cdot 10^{-2} $ \\
            $2.50\cdot 10^{-2}$& $141$ & $1.04$ &$4.36\cdot 10^{-2} $ & $138$& $1.01$ &$9.87\cdot 10^{-3} $ & $138$ & $1.01$&$9.48\cdot 10^{-3} $ \\
            $1.25\cdot 10^{-2}$ & $277$ & $1.02$ &$1.71\cdot 10^{-2} $ & $275$& $1.00$ &$1.53\cdot 10^{-3} $ & $275$ & $1.00$ &$1.47\cdot 10^{-3} $\\
            \bottomrule
        \end{tabular}
        \caption{\Cref{basicalgo}. Test on the triangle with Hamiltonian dependent on $s$, $A=2/3$, $B=1/2$, $C=1$ and $\varepsilon=\Delta x/10$. Columns 2--4 refer to $\Delta t =\Delta x^{5/6}$, columns 5--7 refer to $\Delta t=\Delta x/2$ and columns 8--10 refer to $\Delta t=(\min_{\gamma\in\Ecal^+}\Delta_\gamma x)/12$.}\label{table:TriangleDepAlgo1}
    \end{table}

    Finally, we evaluate \cref{iteralgo} with $\varepsilon = \Delta x/10$ and report the results in \cref{table:TriangleDepAlgo2}. Columns 2–4 correspond to $\Delta t = \Delta x^{5/6}$, columns 5–7 to $\Delta t = \Delta x / 2$, and columns 8–10 to $\Delta t = (\min_{\gamma \in \Ecal^+} \Delta x_\gamma) / 12$. Comparing this table with \cref{table:TriangleDepAlgo1}, we find that \cref{iteralgo} exhibits a similar approximation to \cref{basicalgo} and reduces the number of iterations on average by $67\%$ for $\Delta t = \Delta x^{5/6}$, $80 \%$ for $\Delta t = \Delta x / 2$ and $83\%$ for $\Delta t = (\min_{\gamma \in \Ecal^+} \Delta x_\gamma) / 12$.

    \begin{table}[!ht]
        \renewcommand\arraystretch{1.1}

        \centering
        \begin{tabular}{llllllllll}
            \toprule
            ${\Delta x}$ & $k$ & $c^{\Delta}$ &$|c^{\Delta}-c |$ & $k$ & $c^{\Delta}$ & $|c^{\Delta}-c |$ & $k$ & $c^{\Delta}$& $|c^{\Delta}-c |$ \\
            \midrule
            $2.00\cdot 10^{-1}$ & $4$ & $1.91$ &$9.13\cdot 10^{-1} $ & $7$& $1.18$ &$1.79\cdot 10^{-1} $ & $5$ & $1.15$ &$1.50\cdot 10^{-1} $ \\
            $1.00\cdot 10^{-1}$ & $12$ & $1.29$ &$2.92\cdot 10^{-1} $ & $10$ & $1.07$&$7.09\cdot 10^{-2} $ & $8$ & $1.06$&$6.32\cdot 10^{-2} $ \\
            $5.00\cdot 10^{-2}$ & $48$ & $1.10$ &$9.59\cdot 10^{-2} $ & $15$ & $1.03$&$2.60\cdot 10^{-2} $ & $13$ & $1.02$&$2.42\cdot 10^{-2} $ \\
            $2.50\cdot 10^{-2}$& $56$ & $1.04$ &$4.12\cdot 10^{-2} $ & $16$& $1.01$ &$7.52\cdot 10^{-3} $ & $15$ & $1.01$&$7.40\cdot 10^{-3} $ \\
            $1.25\cdot 10^{-2}$ & $65$ & $1.02$ &$1.59\cdot 10^{-2} $ & $6$& $1.00$ &$5.98\cdot 10^{-4} $ & $7$ & $1.00$ &$5.06\cdot 10^{-4} $\\
            \bottomrule
        \end{tabular}
        \caption{\Cref{iteralgo}. Test on the triangle with Hamiltonian dependent on $s$, $A=2/3$, $B=1/2$, $C=1$ and $\varepsilon=\Delta x/10$. Columns 2--4 refer to $\Delta t =\Delta x^{5/6}$, columns 5--7 refer to $\Delta t=\Delta x/2$ and columns 8--10 refer to $\Delta t=(\min_{\gamma\in\Ecal^+}\Delta_\gamma x)/12$.}\label{table:TriangleDepAlgo2}
    \end{table}

    \subsubsection{Hamiltonian Independent of \texorpdfstring{$s$}{s}}\label{section_tri_indepent}

    We analyze a simpler problem in which the Hamiltonians are independent of $s$. Let us take the Hamiltonians $H_{\gamma_i}: \Rds\longrightarrow \Rds$ for $i=1,2,3$
    \begin{equation*}
        H_{\gamma_1}(\mu)\coloneqq(\mu+1)^2,\qquad H_{\gamma_2}(\mu)\coloneqq\mu^2+B,\qquad H_{\gamma_3}(\mu)\coloneqq(\mu+A)(\mu+2A)+C,
    \end{equation*}
    where $A$, $B$, $C$ are nonnegative constants with $C\ge B$. If $A=\sqrt C-1+\sqrt{C-B}\ge0$, then $c=C$ is the exact critical value. We consider the problem with $A=1$, $B=0$ and $C=1$, for which we can set $\beta_0=9.1$.

    We begin by performing a preliminary analysis of \cref{basicalgo} through two tests, both using $\Delta t=(\min_{\gamma\in\Ecal^+}\Delta_\gamma x)/9.1$. The first test employs $\varepsilon=\Delta x/10$ (columns 2--4 of \cref{table:TriangleIndepAlgo1k2000}), while the second test fixes $2000$ iterations to be executed (\cref{fig:TriIndep} and columns 5–7 of \cref{table:TriangleIndepAlgo1k2000}). As observed in the previous problem, the decreasing behavior of $(\ovc^\Delta_k - \unc^\Delta_k)/2$, shown in \cref{fig:TriIndep} (left), confirms that a smaller tolerance $\varepsilon$ in \cref{basicalgo} leads the algorithm to perform a larger number of iterations before stopping. However, unlike the previous problem where the Hamiltonians depend on $s$, the current tests do not approximate critical values with error of the same order of magnitude. Indeed, as illustrated in \cref{fig:TriIndep} (right), the error $ \mleft| \frac{\ovc^\Delta_k + \unc^\Delta_k}{2} - c \mright| $ decreases as $ k $ increases, without reaching a plateau. In this case, the decreasing trend of errors as $\Delta x$ decreases, presented in column 4 of \cref{table:TriangleIndepAlgo1k2000}, dependents on the tolerance proportional to $\Delta x$. This is because as $\Delta x$ decreases, the tolerance $\varepsilon$ also decreases, requiring to an increase in the number of iterations (see column 2 of \cref{table:TriangleIndepAlgo1k2000}). What has just been observed does not contradict the theoretical analysis conducted at the beginning of \cref{simsec}. Indeed, as derived from \cref{basicerr} (or equivalently, from \cref{basicstop}), the error $ \mleft| \frac{\ovc^\Delta_k + \unc^\Delta_k}{2} - c \mright| $ is bounded above by a term inversely proportional to $ k$ (or equivalently, by $(\ovc_k^\Delta-\unc_k^\Delta)/2$) and an additional term depending on $\min \mleft\{ \frac{\Delta x}{\sqrt{\Delta t}}, \frac{\Delta x^2}{\Delta t^{3/2}} \mright\} $, independent of $ k $. The latter term originates exclusively from the approximation error estimate for the semi-Lagrangian scheme applied to~\eqref{eq:globteik}. However, as thoroughly analyzed in~\cite{CarliniCoscettiPozza24}, when the Hamiltonians are independent of $ s $ and condition~\eqref{eq:goodpair} holds, the approximation error is primarily due to the approximation of the minimum in the scheme. As a result, this error is much smaller than the estimate in \cref{errCS}. Consequently, in the error estimate for $\mleft| \frac{\ovc^\Delta_k + \unc^\Delta_k}{2} - c \mright|$, the dominant term is that one proportional to $1/k$ (or equivalently, the dominant term is $(\ovc_k^\Delta-\unc_k^\Delta)/2$), which explains the absence of a plateau in \cref{fig:TriIndep} (right).

    \begin{figure}[!ht]
        \centering
        \subfigure
        {
            \includegraphics[width=6cm]{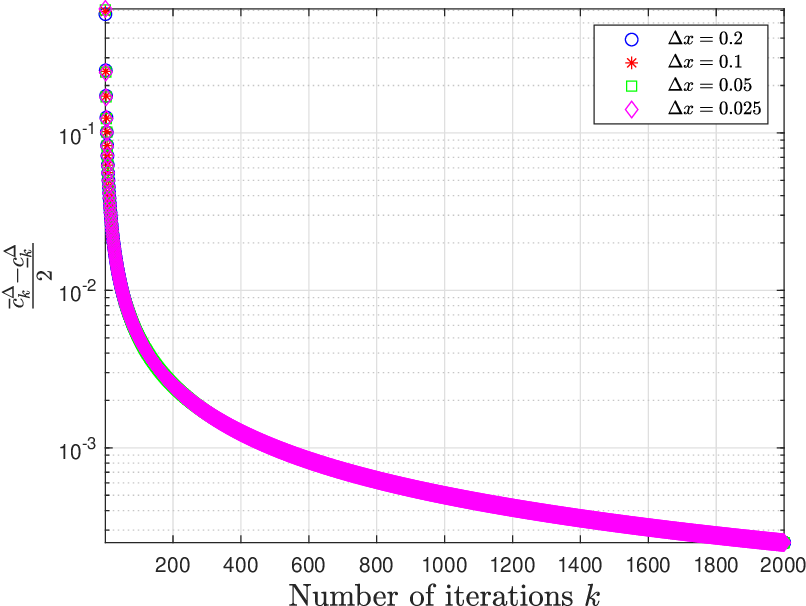}
        }
        \hspace{5mm}
        \subfigure
        {
            \includegraphics[width=6cm]{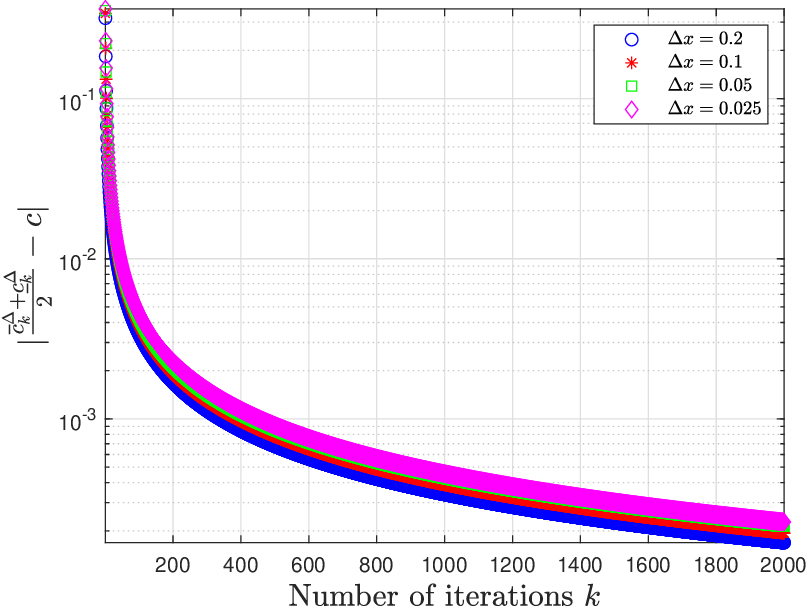}
        }
        \caption{\Cref{basicalgo}. Test on the triangle with Hamiltonian independent of $s$, $A=1$, $B=0$, $C=1$, $\Delta t=(\min_{\gamma\in\Ecal^+}\Delta_\gamma x)/9.1$ and $2000$ fixed iterations.}\label{fig:TriIndep}
    \end{figure}

    \begin{table}[!ht]
        \renewcommand\arraystretch{1.1}
        \centering
        \begin{tabular}{lllllll}
            \toprule
            ${\Delta x}$ & $k$ & $c^{\Delta}$ &$|c^{\Delta}-c |$ & $k$ & $c^{\Delta}$ & $|c^{\Delta}-c |$ \\
            \midrule
            $2.00\cdot 10^{-1}$ & $25$ & $1.01$ &$1.35\cdot 10^{-2} $ & $2000$& $1.00$ &$1.69\cdot 10^{-4} $ \\
            $1.00\cdot 10^{-1}$ & $51$ & $1.01$ &$7.60\cdot 10^{-3} $ & $2000$ & $1.00$&$1.94\cdot 10^{-4} $ \\
            $5.00\cdot 10^{-2}$ & $100$ & $1.00$ &$4.27\cdot 10^{-3} $ & $2000$ & $1.00$&$2.14\cdot 10^{-4} $ \\
            $2.50\cdot 10^{-2}$& $200$ & $1.00$ &$2.28\cdot 10^{-3} $ & $2000$& $1.00$ &$2.28\cdot 10^{-3} $ \\
            $1.25\cdot 10^{-2}$ & $400$ & $1.00$ &$1.18\cdot 10^{-3} $ & $2000$& $1.00$ &$2.37\cdot 10^{-4} $\\
            \bottomrule
        \end{tabular}
        \caption{\Cref{basicalgo}. Test on the triangle with Hamiltonian independent of $s$, $A=1$, $B=0$, $C=1$ and $\Delta t=(\min_{\gamma\in\Ecal^+}\Delta_\gamma x)/9.1$. Columns 2--4 refer to tests with $\varepsilon=\Delta x/10$ and columns 5--7 refer to tests with $2000$ fixed iterations.}\label{table:TriangleIndepAlgo1k2000}
    \end{table}

    Next, we implement \cref{basicalgo} with the same time step $ \Delta t = (\min_{\gamma \in \mathcal{E}^+} \Delta_{\gamma} x) / 9.1 $, but with a smaller tolerance $ \varepsilon = \Delta x / 100 $ (columns 8–11 in \cref{table:TriangleIndepAlgo1}). This yields smaller errors compared to the previous test shown in column 4 of \cref{table:TriangleIndepAlgo1k2000}. As just discussed, this reduction in error can be attributed to the choice of a smaller tolerance. Additionally, the error exhibits a decreasing trend as $\Delta x $ decreases, which is still caused by the tolerance being proportional to the spatial step.

    Furthermore, we apply \cref{basicalgo} keeping $ \varepsilon = \Delta x/100 $ but exploring different time step choices: $\Delta t = \Delta x^{5/6} $ (columns 2–4 of \cref{table:TriangleIndepAlgo1}) and $\Delta t = \Delta x / 2 $ (columns 5–7 of \cref{table:TriangleIndepAlgo1}). It can be seen that the largest error occurs when $ \Delta t = \Delta x^{5/6} $. On the other hand, setting $ \Delta t = \Delta x / 2 $ results in errors of a comparable magnitude while preserving an almost identical number of iterations. This latter choice is particularly advantageous, as the larger time step leads to a reduction in computational cost.

    \begin{table}[!ht]
        \renewcommand\arraystretch{1.1}
        \centering
        \begin{tabular}{llllllllll}
            \toprule
            ${\Delta x}$ & $k$ & $c^{\Delta}$ &$|c^{\Delta}-c |$ & $k$ & $c^{\Delta}$ & $|c^{\Delta}-c |$ & $k$ & $c^{\Delta}$& $|c^{\Delta}-c |$ \\
            \midrule
            $2.00\cdot 10^{-1}$ & $263$ & $0.95$ &$4.93\cdot 10^{-2} $ & $250$& $1.00$ &$1.64\cdot 10^{-3} $ & $251$ & $1.00$ &$1.35\cdot 10^{-3} $ \\
            $1.00\cdot 10^{-1}$ & $504$ & $0.99$ &$9.70\cdot 10^{-3} $ & $501$ & $1.00$&$8.78\cdot 10^{-4} $ & $501$ & $1.00$&$7.74\cdot 10^{-4} $ \\
            $5.00\cdot 10^{-2}$ & $1004$ & $1.00$ &$4.51\cdot 10^{-3} $ & $1000$ & $1.00$&$4.62\cdot 10^{-4} $ & $1001$ & $1.00$&$4.27\cdot 10^{-4} $ \\
            $2.50\cdot 10^{-2}$& $2004$ & $1.00$ &$1.77\cdot 10^{-3} $ & $2001$& $1.00$ &$2.39\cdot 10^{-4} $ & $2001$ & $1.00$&$2.27\cdot 10^{-4} $ \\
            $1.25\cdot 10^{-2}$ & $4003$ & $1.00$ &$7.50\cdot 10^{-4} $ & $4000$& $1.00$ &$1.22\cdot 10^{-4} $ & $4001$ & $1.00$ &$1.18\cdot 10^{-4} $\\
            \bottomrule
        \end{tabular}
        \caption{\Cref{basicalgo}. Test on the triangle with Hamiltonian independent on $s$, $A=1$, $B=0$, $C=1$ and $\varepsilon=\Delta x/100$. Columns 2--4 refer to $\Delta t =\Delta x^{5/6}$, columns 5--7 refer to $\Delta t=\Delta x/2$ and columns 8--10 refer to $\Delta t=(\min_{\gamma\in\Ecal^+}\Delta_\gamma x)/9.1$.}\label{table:TriangleIndepAlgo1}
    \end{table}

    Finally, we implement \cref{iteralgo} with $ \varepsilon = \Delta x/100 $ and report the results in \cref{table:TriangleIndepAlgo2}, whose columns 2–4 refer to $\Delta t = (\Delta x)^{5/6}$, columns 5–7 to $\Delta t = \Delta x / 2$, and columns 8–10 to $\Delta t = \min_{\gamma \in \Ecal^+} \Delta x_\gamma / 9.1$. A comparison with \cref{table:TriangleIndepAlgo1} reveals that \cref{iteralgo} achieves a better approximation with substantially fewer iterations. More specifically, the number of iterations is reduced on average by $79 \%$ for $\Delta t = (\Delta x)^{5/6}$, $93 \%$ for $\Delta t = \Delta x / 2$ and $97\%$ for $\Delta t = \min_{\gamma \in \Ecal^+} \Delta x_\gamma / 9.1$.

    \begin{table}[!ht]
        \renewcommand\arraystretch{1.1}
        \centering
        \begin{tabular}{llllllllll}
            \toprule
            ${\Delta x}$ & $k$ & $c^{\Delta}$ &$|c^{\Delta}-c |$ & $k$ & $c^{\Delta}$ & $|c^{\Delta}-c |$ & $k$ & $c^{\Delta}$& $|c^{\Delta}-c |$ \\
            \midrule
            $2.00\cdot 10^{-1}$ & $49$ & $0.95$ &$5.11\cdot 10^{-2} $ & $18$& $1.00$ &$1.41\cdot 10^{-3} $ & $9$ & $1.00$ &$5.11\cdot 10^{-5} $ \\
            $1.00\cdot 10^{-1}$ & $73$ & $0.99$ &$1.07\cdot 10^{-2} $ & $36$ & $1.00$&$4.77\cdot 10^{-5} $ & $17$ & $1.00$&$3.48\cdot 10^{-5} $ \\
            $5.00\cdot 10^{-2}$ & $168$ & $0.99$ &$5.00\cdot 10^{-3} $ & $72$ & $1.00$&$1.35\cdot 10^{-5} $ & $34$ & $1.00$&$1.04\cdot 10^{-5} $ \\
            $2.50\cdot 10^{-2}$& $442$ & $1.00$ &$2.02\cdot 10^{-3} $ & $143$& $1.00$ &$3.66\cdot 10^{-6} $ & $68$ & $1.00$&$3.11\cdot 10^{-6} $ \\
            $1.25\cdot 10^{-2}$ & $1358$ & $1.00$ &$8.74\cdot 10^{-4} $ & $286$& $1.00$ &$9.36\cdot 10^{-7} $ & $136$ & $1.00$ &$8.34\cdot 10^{-7} $\\
            \bottomrule
        \end{tabular}
        \caption{\Cref{iteralgo}. Test on the triangle with Hamiltonian independent on $s$, $A=1$, $B=0$, $C=1$ and $\varepsilon=\Delta x/100$. Columns 2--4 refer to $\Delta t =\Delta x^{5/6}$, columns 5--7 refer to $\Delta t=\Delta x/2$ and columns 8--10 refer to $\Delta t=(\min_{\gamma\in\Ecal^+}\Delta_\gamma x)/9.1$.}\label{table:TriangleIndepAlgo2}
    \end{table}

    \subsection{Test on a Traffic Circle}

    We consider as $\Gamma \subset \Rds^2$ a network formed by 8 vertices and 12 arcs. The vertices are defined as
    \begin{align*}
        z_1\coloneqq\;&(-2, 0), & z_2\coloneqq\;&(-1, 0), & z_3\coloneqq\;&(0,2), & z_4\coloneqq\;&(0,1),\\
        z_5\coloneqq\;&(2, 0), & z_6\coloneqq\;&(1, 0), & z_7\coloneqq\;&(0,-2), & z_8\coloneqq\;&(0,-1).
    \end{align*}
    The arcs $\gamma_i:[0,1]\longrightarrow \Rds^2$ for $i=1,\dots, 12$ are defined as
    \begin{align*}
        \gamma_1(s)\coloneqq\;&(1-s)z_1+sz_2, & \gamma_2(s)\coloneqq\;&(1-s)z_1+sz_3,& \gamma_3(s)\coloneqq\;&(1-s)z_1+sz_7,\\
        \gamma_4(s)\coloneqq\;&(1-s)z_2+sz_4, & \gamma_5(s)\coloneqq\;&(1-s)z_2+sz_8,& \gamma_6(s)\coloneqq\;&(1-s)z_3+sz_4,\\
        \gamma_7(s)\coloneqq\;&(1-s)z_3+sz_5, & \gamma_8(s)\coloneqq\;&(1-s)z_4+sz_6,& \gamma_9(s)\coloneqq\;&(1-s)z_5+sz_6,\\
        \gamma_{10}(s)\coloneqq\;&(1-s)z_5+sz_7, & \gamma_{11}(s)\coloneqq\;&(1-s)z_6+sz_8,& \gamma_{12}(s)\coloneqq\;&(1-s)z_7+sz_8,
    \end{align*}
    plus the reversed arcs (\cref{Network:TrafficCircle}).

    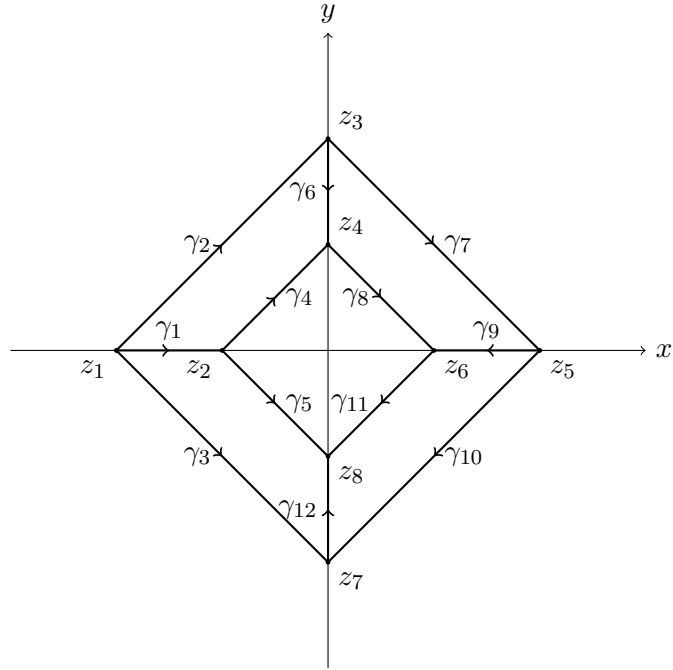
\begin{figure}[!ht]
        \centering
        \begin{tikzpicture}[scale=1.4]
            \draw[->] (-3,0) -- (3,0) node[right] {$x$};
            \draw[->] (0,-3) -- (0,3) node[above] {$y$};

            \node[below left] at (-2,0) {$z_1$};
            \node[below left] at (-1,0) {$z_2$};
            \node[above right] at (0,2) {$z_3$};
            \node[above right] at (0,1) {$z_4$};
            \node[below right] at (2,0) {$z_5$};
            \node[below right] at (1,0) {$z_6$};
            \node[below right] at (0,-2) {$z_7$};
            \node[below right] at (0,-1) {$z_8$};

            \draw[thick, postaction={decorate, decoration={markings, mark=at position 0.5 with {\arrow{>}}}}] (-2,0) -- (-1,0)
            node[midway, above ] {$\gamma_1$};

            \draw[thick, postaction={decorate, decoration={markings, mark=at position 0.5 with {\arrow{>}}}}] (-2,0) -- (0,2)
            node[midway, left ] {$\gamma_2$};

            \draw[thick, postaction={decorate, decoration={markings, mark=at position 0.5 with {\arrow{>}}}}] (-2,0) -- (0,-2)
            node[midway, left ] {$\gamma_3$};

            \draw[thick, postaction={decorate, decoration={markings, mark=at position 0.5 with {\arrow{>}}}}] (-1,0) -- (0,1)
            node[midway, right ] {$\gamma_4$};

            \draw[thick, postaction={decorate, decoration={markings, mark=at position 0.5 with {\arrow{>}}}}] (-1,0) -- (0,-1)
            node[midway, right ] {$\gamma_5$};

            \draw[thick, postaction={decorate, decoration={markings, mark=at position 0.5 with {\arrow{>}}}}] (0,2) -- (0,1)
            node[midway, left ] {$\gamma_6$};

            \draw[thick, postaction={decorate, decoration={markings, mark=at position 0.5 with {\arrow{>}}}}] (0,2) -- (2,0)
            node[midway, right ] {$\gamma_7$};

            \draw[thick, postaction={decorate, decoration={markings, mark=at position 0.5 with {\arrow{>}}}}] (0,1) -- (1,0)
            node[midway, left ] {$\gamma_8$};

            \draw[thick, postaction={decorate, decoration={markings, mark=at position 0.5 with {\arrow{>}}}}] (2,0) -- (1,0)
            node[midway, above ] {$\gamma_9$};

            \draw[thick, postaction={decorate, decoration={markings, mark=at position 0.5 with {\arrow{>}}}}] (2,0) -- (0,-2)
            node[midway, right ] {$\gamma_{10}$};

            \draw[thick, postaction={decorate, decoration={markings, mark=at position 0.5 with {\arrow{>}}}}] (1,0) -- (0,-1)
            node[midway, left ] {$\gamma_{11}$};

            \draw[thick, postaction={decorate, decoration={markings, mark=at position 0.5 with {\arrow{>}}}}] (0,-2) -- (0,-1)
            node[midway, left ] {$\gamma_{12}$};

            \filldraw[black] (-2,0) circle (0.5pt);
            \filldraw[black] (-1,0) circle (0.5pt);
            \filldraw[black] (0,2) circle (0.5pt);
            \filldraw[black] (0,1) circle (0.5pt);
            \filldraw[black] (2,0) circle (0.5pt);
            \filldraw[black] (1,0) circle (0.5pt);
            \filldraw[black] (0,-2) circle (0.5pt);
            \filldraw[black] (0,-1) circle (0.5pt);
        \end{tikzpicture}
        \caption{Network $\Gamma \subset \Rds^2 $: the traffic circle.}\label{Network:TrafficCircle}
    \end{figure}

    \subsubsection{Hamiltonian Dependent on \texorpdfstring{$s$}{s}}

    Let us take the Hamiltonians $H_{\gamma_i}:[0,1]\times \Rds\longrightarrow \Rds$ for $i=1,\dots,12$
    \begin{equation*}
        H_{\gamma_i}(s,\mu)\coloneqq\mleft\{
        \begin{aligned}
            &\frac{(\mu-1)^2}{2}- 5, && \text{for } i=2,3,7,10 ,\\
            &\frac{(\mu+4s)^2}{4}, && \text{for } i=4,5,8,11 , \\
            &\frac{\mu^2}{2}- 5, && \text{for } i=1,6,9,12,
        \end{aligned}
        \mright.
    \end{equation*}
    where $\beta_0 = 9.5$ can be chosen.

    Since the exact critical value for this problem is not explicitly known, we take $\hat{c} = 2.59 \cdot 10^{-1}$ as the reference critical value, computed by running \cref{iteralgo} with parameters $\Delta x = 1.25 \cdot 10^{-2}$, $\Delta t = (\min_{\gamma\in\Ecal^+} \Delta_\gamma x)/9.5$ and $ \varepsilon = \Delta x/10$, and rounding to the first three significant figures.

    We carry out the study of the algorithms, starting with \cref{basicalgo}. The algorithm is tested with $\Delta t = (\min_{\gamma\in\Ecal^+} \Delta_\gamma x)/9.5$ and with $2000$ fixed iterations to be executed. The corresponding results are presented in \cref{table:RotDepAlgo1k2000} and \cref{fig:RotDep}. We compare these results with those in columns 8--10 of \cref{table:RotDepAlgo1}, obtained by implementing the algorithm with the same time step $\Delta t = (\min_{\gamma\in\Ecal^+} \Delta_\gamma x)/9.5$ but setting the tolerance $\varepsilon = \Delta x/10$. The observed errors are approximately the same. As illustrated in \cref{fig:RotDep}, the sequences $(\ovc^\Delta_k - \unc^\Delta_k)/2$ decrease in $k$ and the errors $\mleft| \frac{\ovc^\Delta_k + \unc^\Delta_k}{2} - \hat c \mright|$ reach a plateau, whose value decreases in $\Delta x$. This confirms that choosing $\varepsilon=\Delta x/10$ is not restrictive, as already analyzed in \cref{section:TriDependent}.

    \begin{figure}[!ht]
        \centering
        \subfigure
        {
            \includegraphics[width=4.5cm]{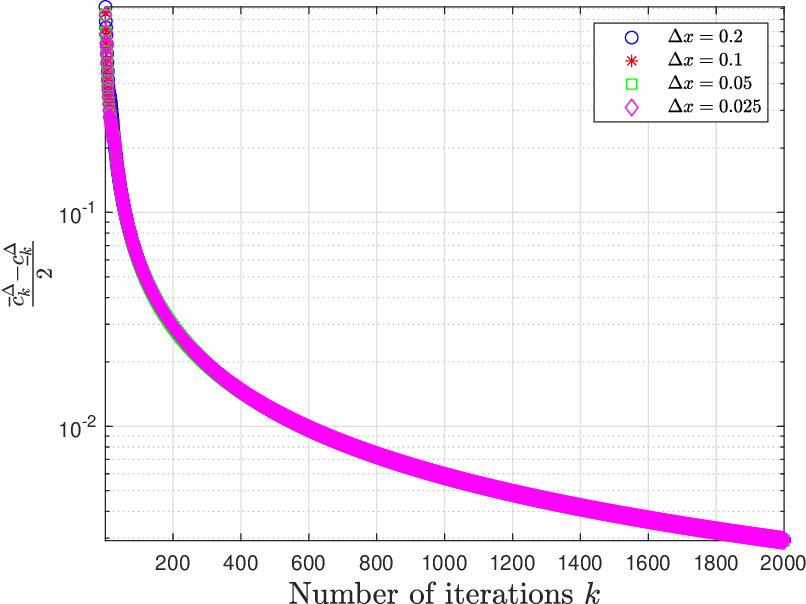}
        }
        \hspace{5mm}
        \subfigure
        {
            \includegraphics[width=4.5cm]{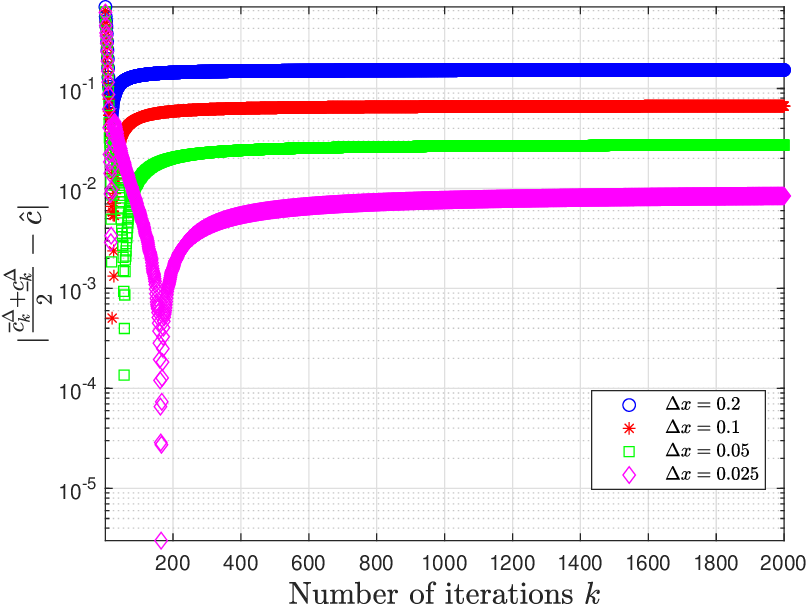}
        }
        \subfigure
        {
            \includegraphics[width=4.15cm]{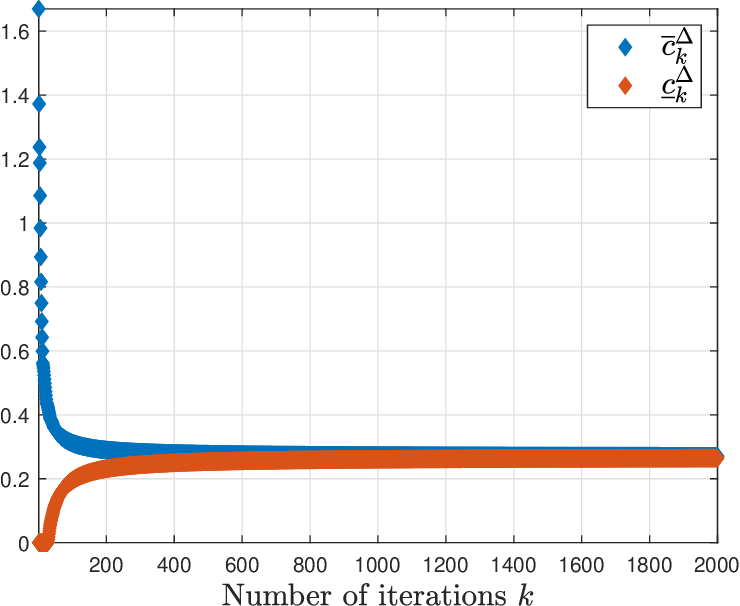}
        }
        \caption{\Cref{basicalgo}. Test on the traffic circle with Hamiltonian dependent on $s$, $\Delta t=(\min_{\gamma\in\Ecal^+}\Delta_\gamma x)/9.5$ and $2000$ fixed iterations. The third graph refers to the test with $\Delta x=2.50\cdot 10^{-2}$.}\label{fig:RotDep}
    \end{figure}

    \begin{table}[!ht]
        \renewcommand\arraystretch{1.1}
        \centering
        \begin{tabular}{lllllll}
            \toprule
            ${\Delta x}$ & $k$ & $c^{\Delta}$ &$|c^{\Delta}-\hat c |$ \\
            \midrule
            $2.00\cdot 10^{-1}$ & $2000$& $0.41$ &$1.53\cdot 10^{-1} $ \\
            $1.00\cdot 10^{-1}$ & $2000$ & $0.33$&$6.67\cdot 10^{-2} $ \\
            $5.00\cdot 10^{-2}$ & $2000$ & $0.29$&$2.72\cdot 10^{-2} $ \\
            $2.50\cdot 10^{-2}$& $2000$& $0.27$ &$8.42\cdot 10^{-3} $ \\
            \bottomrule
        \end{tabular}
        \caption{\Cref{basicalgo}. Test on the traffic circle with Hamiltonian dependent on $s$, $\Delta t=(\min_{\gamma\in\Ecal^+}\Delta_\gamma x)/9.5$ and $2000$ fixed iterations.}\label{table:RotDepAlgo1k2000}
    \end{table}

    \begin{table}[!ht]
        \renewcommand\arraystretch{1.1}
        \centering
        \begin{tabular}{llllllllll}
            \toprule
            ${\Delta x}$ & $k$ & $c^{\Delta}$ &$|c^{\Delta}-\hat c |$ & $k$ & $c^{\Delta}$ & $|c^{\Delta}-\hat c |$ & $k$ & $c^{\Delta}$& $|c^{\Delta}-\hat c |$ \\
            \midrule
            $2.00\cdot 10^{-1}$ & $334$ & $0.50$ &$2.36\cdot 10^{-1} $ & $294$& $0.41$ &$1.48\cdot 10^{-1} $ & $294$ & $0.41$ &$1.48\cdot 10^{-1} $ \\
            $1.00\cdot 10^{-1}$ & $613$ & $0.33$&$7.43\cdot 10^{-2} $ & $584$ & $0.32$&$6.47\cdot 10^{-2} $ & $584$ & $0.32$ &$6.46\cdot 10^{-2} $ \\
            $5.00\cdot 10^{-2}$ & $1192$ & $0.29$&$2.67\cdot 10^{-2} $ & $1168$ & $0.29$&$2.67\cdot 10^{-2} $ & $1168$ & $0.29$ &$2.66\cdot 10^{-2} $ \\
            $2.50\cdot 10^{-2}$ & $2372$ & $0.27$&$8.91\cdot 10^{-3} $ & $2341$& $0.27$ &$8.54\cdot 10^{-3} $ & $2341$ & $0.27$ &$8.53\cdot 10^{-3} $ \\
            $1.25\cdot 10^{-2}$& $4711$ & $0.26$ &$1.08\cdot 10^{-3} $ & $4679$& $0.26$ & $3.85\cdot 10^{-4} $ & $4679$ & $0.26$&$ 3.87\cdot 10^{-4} $\\
            \bottomrule
        \end{tabular}
        \caption{\Cref{basicalgo}. Test on the traffic circle with Hamiltonian dependent on $s$ and $\varepsilon=\Delta x/10$. Columns 2--4 refer to $\Delta t =\Delta x^{5/6}$, columns 5--7 refer to $\Delta t=\Delta x/2$ and columns 8--10 refer to $\Delta t=(\min_{\gamma\in\Ecal^+}\Delta_\gamma x)/9.5$.}\label{table:RotDepAlgo1}
    \end{table}

    We point out that, while \cref{fig:RotDep} (right) shows decreasing errors in $k$ toward a plateau, \cref{fig:RotDep} (center) reports a different behavior. For instance, in the test with $\Delta x= 2.50\cdot 10^{-2}$, the error reaches the lowest point at iteration 165 before converging and stabilizing at a certain value. In this test, when $k=165$, the distance between $\ovc^\Delta_k$ and $\unc^\Delta_k$ is not smaller than $2\varepsilon$, as can be seen in \cref{fig:RotDep} (left and right), however their average is very close to the reference critical value.

    We continue our study setting $\varepsilon=\Delta x/10$, $\Delta t = \Delta x^{5/6} $, $\Delta t = \Delta x / 2 $, $\Delta t = (\min_{\gamma\in\Ecal^+} \Delta_\gamma x)/9.5$ and comparing the results given by \cref{basicalgo} and \cref{iteralgo}, shown in \cref{table:RotDepAlgo1} and \cref{table:RotDepAlgo2}, respectively. What stands out is that \cref{iteralgo} requires fewer iterations than \cref{basicalgo}. In fact, \cref{iteralgo} yields an average reduction in the number of iterations as following: $67\%$ for $\Delta t = \Delta x^{5/6}$ (column 2 of \cref{table:RotDepAlgo1} and \cref{table:RotDepAlgo2}), $84 \%$ for $\Delta t = \Delta x / 2 $ (column 5 of \cref{table:RotDepAlgo1} and \cref{table:RotDepAlgo2}) and $87 \%$ for $\Delta t = (\min_{\gamma\in\Ecal^+} \Delta_\gamma x)/9.5$ (column 8 of \cref{table:RotDepAlgo1} and \cref{table:RotDepAlgo2}).

    \begin{table}[!ht]
        \renewcommand\arraystretch{1.1}
        \centering
        \begin{tabular}{llllllllll}
            \toprule
            ${\Delta x}$ & $k$ & $c^{\Delta}$ &$|c^{\Delta}-\hat c |$ & $k$ & $c^{\Delta}$ & $|c^{\Delta}-\hat c |$ & $k$ & $c^{\Delta}$& $|c^{\Delta}-\hat c |$ \\
            \midrule
            $2.00\cdot 10^{-1}$ & $99$ & $0.50$ &$2.43\cdot 10^{-1} $ & $47$& $0.41$ & $1.54\cdot 10^{-1} $ & $38$ & $0.41$ &$1.54\cdot 10^{-1} $ \\
            $1.00\cdot 10^{-1}$ & $190$ & $0.34$ &$7.76\cdot 10^{-2} $ & $93$ & $0.33$ & $6.78\cdot 10^{-2} $ & $79$ & $0.33$&$6.76\cdot 10^{-2} $ \\
            $5.00\cdot 10^{-2}$ & $435$ & $0.29$ &$2.85\cdot 10^{-2} $ & $184$ & $0.29$ &$2.80\cdot 10^{-2} $ & $159$ & $0.29$&$2.80\cdot 10^{-2} $ \\
            $2.50\cdot 10^{-2}$& $814$ & $0.27$ &$1.02\cdot 10^{-2} $ & $368$& $0.27$ & $9.21\cdot 10^{-3} $ & $319$ & $0.27$&$9.20\cdot 10^{-3} $ \\
            $1.25\cdot 10^{-2}$& $1576$ & $0.26$ &$1.36\cdot 10^{-3} $ & $737$& $0.26$ & $5.25\cdot 10^{-5} $ & $643$ & $0.26$&$6.17\cdot 10^{-5} $\\
            \bottomrule
        \end{tabular}
        \caption{\Cref{iteralgo}. Test on the traffic circle with Hamiltonian dependent on $s$ and $\varepsilon=\Delta x/10$. Columns 2--4 refer to $\Delta t =\Delta x^{5/6}$, columns 5--7 refer to $\Delta t=\Delta x/2$ and columns 8--10 refer to $\Delta t=(\min_{\gamma\in\Ecal^+}\Delta_\gamma x)/9.5$.}\label{table:RotDepAlgo2}
    \end{table}

    \subsubsection{Hamiltonian Independent of \texorpdfstring{$s$}{s}}

    We now simplify the previous test and change the Hamiltonian on the inner traffic circle so that it becomes independent of $s$. Let us take the Hamiltonians $H_{\gamma_i}: \Rds\longrightarrow \Rds$ for $i=1,\dots,12$

    \begin{equation*}
        H_{\gamma_i}(\mu)\coloneqq\mleft\{
        \begin{aligned}
            &\frac{(\mu-1)^2}{2}- 5, && \text{for } i=2,3,7,10 ,\\
            &\frac{(\mu+2)^2}{2}- 2, && \text{for } i=4,5,8,11 , \\
            &\frac{\mu^2}{2}- 5, && \text{for } i=1,6,9,12.
        \end{aligned}
        \mright.
    \end{equation*}

    For this problem we can choose $\beta_0=7.5$ and, as in the previous case, the exact critical value is unknown. We set the reference critical value $\hat{c} = -1.50$ as the value computed by \cref{iteralgo} with $\Delta x = 1.25 \cdot 10^{-2}$, $\Delta t=(\min_{\gamma\in\Ecal^+}\Delta_\gamma x)/7.5$, $ \varepsilon = \Delta x/10$ and rounded to the first three significant figures.

    We test \cref{basicalgo} with $\Delta t=(\min_{\gamma\in\Ecal^+} \Delta_\gamma x)/7.5$ and $2000$ fixed iterations to be executed (\cref{table:RotIndepAlgo1k2000} and \cref{fig:RotIndep}). The current problem is independent of $s$ and, as detailed in \cref{section_tri_indepent}, the approximation error introduced by the semi-Lagrangian scheme is smaller than the estimate in \cref{errCS}. As a result, the errors $\mleft| \frac{\ovc^\Delta_k + \unc^\Delta_k}{2} - \hat c \mright|$ decrease and do not reach a plateau, as shown in \cref{fig:RotIndep}. Since $(\ovc^\Delta_k - \unc^\Delta_k)/2$ decreases in $k$ (see \cref{fig:RotIndep}), when we implement \cref{basicalgo} with $\varepsilon=\Delta x/10$ (columns 8--10 of \cref{table:RotIndepAlgo1}), the decreasing trend of the errors as $\Delta x$ decreases is a consequence of the chosen tolerance.

    \begin{figure}[!ht]
        \centering
        \subfigure
        {
            \includegraphics[width=5cm]{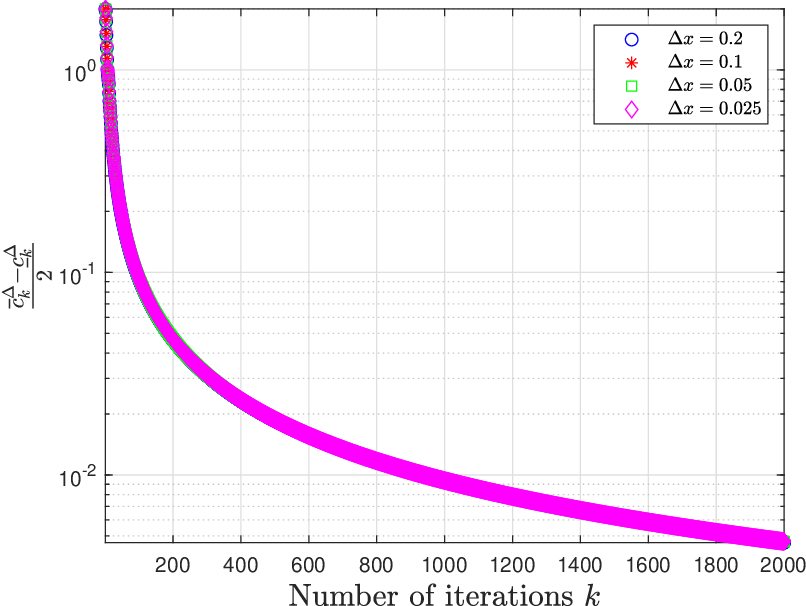}
        }
        \hspace{5mm}
        \subfigure
        {
            \includegraphics[width=5cm]{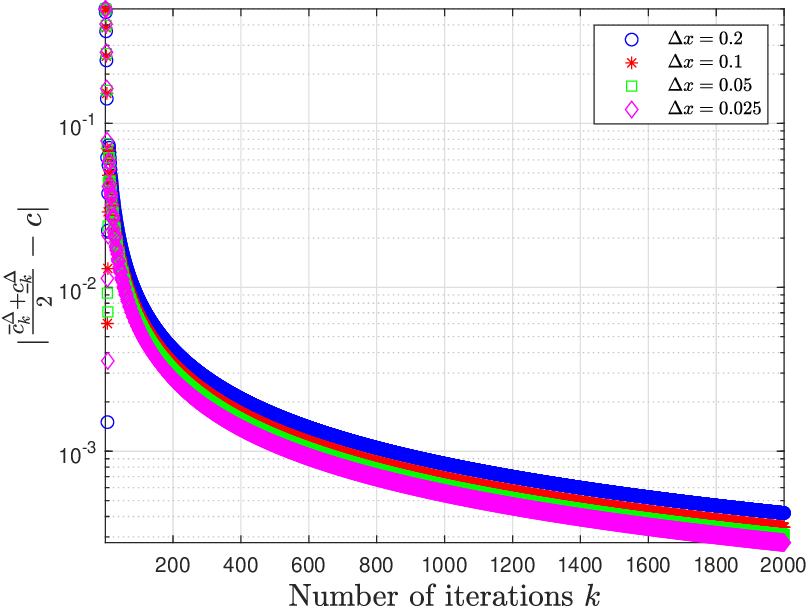}
        }
        \caption{\Cref{basicalgo}. Test on the traffic circle with Hamiltonian independent of $s$, $\Delta t=(\min_{\gamma\in\Ecal^+}\Delta_\gamma x)/7.5$ and $2000$ fixed iterations.}\label{fig:RotIndep}
    \end{figure}

    \begin{table}[!ht]
        \renewcommand\arraystretch{1.1}
        \centering
        \begin{tabular}{lllllll}
            \toprule
            ${\Delta x}$ & $k$ & $c^{\Delta}$ &$|c^{\Delta}-\hat c |$ \\
            \midrule
            $2.00\cdot 10^{-1}$ & $2000$& $-1.50$ &$4.20\cdot 10^{-4} $ \\
            $1.00\cdot 10^{-1}$ & $2000$ & $-1.50$&$3.44\cdot 10^{-4} $ \\
            $5.00\cdot 10^{-2}$ & $2000$ & $-1.50$&$3.10\cdot 10^{-4} $ \\
            $2.50\cdot 10^{-2}$& $2000$& $-1.50$ &$2.77\cdot 10^{-4} $ \\
            \bottomrule
        \end{tabular}
        \caption{\Cref{basicalgo}. Test on the traffic circle with Hamiltonian independent of $s$, $\Delta t=(\min_{\gamma\in\Ecal^+}\Delta_\gamma x)/7.5$ and $2000$ fixed iterations.}\label{table:RotIndepAlgo1k2000}
    \end{table}

    We conduct some tests with $\varepsilon = \Delta x / 10$ and compare the outcomes of \cref{basicalgo} and \cref{iteralgo}, which are presented in \cref{table:RotIndepAlgo1} and \cref{table:RotIndepAlgo2}, respectively. The tests are performed with three different values of $\Delta t$, given by $\Delta t = \Delta x^{5/6}$, $\Delta t = \Delta x / 2$ and $\Delta t = (\min_{\gamma \in \Ecal^+} \Delta_\gamma x)/7.5$. Numerical simulations show that, by implementing \cref{iteralgo} instead of \cref{basicalgo}, we achieve an average reduction in the number of iterations by: $66\%$ for $\Delta t = \Delta x^{5/6}$ (column 2 of \cref{table:RotIndepAlgo1} and \cref{table:RotIndepAlgo2}), $83 \%$ for $\Delta t = \Delta x / 2$ (column 5 of \cref{table:RotIndepAlgo1} and \cref{table:RotIndepAlgo2}) and $88 \%$ for $\Delta t = (\min_{\gamma \in \Ecal^+} \Delta_\gamma x)/7.5$ (column 8 of \cref{table:RotIndepAlgo1} and \cref{table:RotIndepAlgo2}).

    \begin{table}[!ht]
        \renewcommand\arraystretch{1.1}
        \centering
        \resizebox{\textwidth}{!}{
            \begin{tabular}{llllllllll}
                \toprule
                ${\Delta x}$ & $k$ & $c^{\Delta}$ &$|c^{\Delta}-\hat c |$ & $k$ & $c^{\Delta}$ & $|c^{\Delta}-\hat c |$ & $k$ & $c^{\Delta}$& $|c^{\Delta}-\hat c |$ \\
                \midrule
                $2.00\cdot 10^{-1}$ & $244$ & $-1.51$ &$1.23\cdot 10^{-2} $ & $232$& $-1.50$ &$2.97\cdot 10^{-3} $ & $232$ & $-1.50$ &$3.62\cdot 10^{-3} $ \\
                $1.00\cdot 10^{-1}$ & $479$ & $-1.51$&$5.73\cdot 10^{-3} $ & $465$ & $-1.50$&$1.26\cdot 10^{-3} $ & $465$ & $-1.50$ &$1.48\cdot 10^{-3} $ \\
                $5.00\cdot 10^{-2}$ & $952$ & $-1.50$&$2.84\cdot 10^{-3} $ & $938$ & $-1.50$&$5.94\cdot 10^{-4} $ & $938$ & $-1.50$ &$6.61\cdot 10^{-4} $ \\
                $2.50\cdot 10^{-2}$ & $1894$ & $-1.50$&$9.82\cdot 10^{-4} $ & $1878$& $-1.50$ &$2.75\cdot 10^{-4} $ & $1878$ & $-1.50$ &$2.95\cdot 10^{-4} $ \\
                $1.25\cdot 10^{-2}$ & $3774$ & $-1.50$&$4.29\cdot 10^{-4} $ & $3756$& $-1.50$ &$1.30\cdot 10^{-4} $ & $3756$ & $-1.50$ &$1.36\cdot 10^{-4} $ \\
                \bottomrule
            \end{tabular}
        }
        \caption{\Cref{basicalgo}. Test on the traffic circle with Hamiltonian independent of $s$ and $\varepsilon=\Delta x/10$. Columns 2--4 refer to $\Delta t =\Delta x^{5/6}$, columns 5--7 refer to $\Delta t=\Delta x/2$ and columns 8--10 refer to $\Delta t=(\min_{\gamma\in\Ecal^+}\Delta_\gamma x)/7.5$.}\label{table:RotIndepAlgo1}
    \end{table}

    \begin{table}[!ht]
        \renewcommand\arraystretch{1.1}
        \centering
        \resizebox{\textwidth}{!}{
            \begin{tabular}{llllllllll}
                \toprule
                ${\Delta x}$ & $k$ & $c^{\Delta}$ &$|c^{\Delta}-\hat c |$ & $k$ & $c^{\Delta}$ & $|c^{\Delta}-\hat c |$ & $k$ & $c^{\Delta}$& $|c^{\Delta}-\hat c |$ \\
                \midrule
                $2.00\cdot 10^{-1}$ & $76$ & $-1.51$ &$9.31\cdot 10^{-3} $ & $40$& $-1.50$ & $7.33\cdot 10^{-4} $ & $29$ & $-1.50$ &$6.66\cdot 10^{-4} $ \\
                $1.00\cdot 10^{-1}$ & $126$ & $-1.50$ &$4.36\cdot 10^{-3} $ & $77$ & $-1.50$ & $2.19\cdot 10^{-4} $ & $55$ & $-1.50$&$3.20\cdot 10^{-5} $ \\
                $5.00\cdot 10^{-2}$ & $271$ & $-1.50$ &$2.06\cdot 10^{-3} $ & $153$ & $-1.50$ &$4.24\cdot 10^{-5} $ & $109$ & $-1.50$&$8.55\cdot 10^{-6} $ \\
                $2.50\cdot 10^{-2}$& $717$ & $-1.50$ &$6.91\cdot 10^{-4} $ & $306$& $-1.50$ & $1.21\cdot 10^{-5} $ & $218$ & $-1.50$&$1.19\cdot 10^{-5} $ \\
                $1.25\cdot 10^{-2}$& $1656$ & $-1.50$ &$2.18\cdot 10^{-4} $ & $616$& $-1.50$ & $2.94\cdot 10^{-6} $ & $437$ & $-1.50$&$5.15\cdot 10^{-7} $\\
                \bottomrule
            \end{tabular}
        }
        \caption{\Cref{iteralgo}. Test on the traffic circle with Hamiltonian independent of $s$ and $\varepsilon=\Delta x/10$. Columns 2--4 refer to $\Delta t =\Delta x^{5/6}$, columns 5--7 refer to $\Delta t=\Delta x/2$ and columns 8--10 refer to $\Delta t=(\min_{\gamma\in\Ecal^+}\Delta_\gamma x)/7.5$.}\label{table:RotIndepAlgo2}
    \end{table}

    \appendix

    \section{Proof of Theorem~\ref{ltb}}\label{ltbexp}

    \Cref{ltb} has been proved in~\cite[Theorem 5.2]{Pozza25_2} under the additional assumption
    \begin{enumerate}[label=\textbf{(D)},format=\normalfont]
        \item for any $\gamma\in\Ecal$ with $a_\gamma=c=a_0$ the map $s\mapsto\min\limits_{p\in\Rds}H_\gamma(s,p)$ is constant in $[0,|\gamma|]$.
    \end{enumerate}
    This assumption relies on knowing the value of $c$ in advance, which is clearly too restrictive for our purposes. Under this condition the extended Aubry set $\wtAcal_\Gamma$ consists of the support of a collection of arcs, while in our case it could even be made up of a countable number of disjointed pieces of arcs or just a single point. However, the proof of~\cite[Theorem 5.2]{Pozza25_2} can be readily extended to our case using the same methods, with mostly straightforward modifications. Thus, we will just sketch the proof of \cref{ltb}, highlighting the main steps.

    First, the same arguments used in the proof of~\cite[Proposition 5.7]{Pozza25_2} prove the following result.

    \begin{prop}\label{weakltb}
        Let $w$ be a subsolution to \hrefc{}, then $(x,t)\mapsto w(x)-ct$ is a subsolution to~\eqref{eq:globteik}. If $w$ is also a solution to \hrefc{} in $\Gamma\setminus\mleft(\wtAcal_\Gamma\setminus\Acal_\Gamma\mright)$, then $w-ct=\Scal(t)w$ on $\Gamma\times\Rds^+$.
    \end{prop}

    Setting $u$ as in~\cref{eq:ltb.1} and
    \begin{equation*}
        w(x)\coloneqq\min\limits_{y\in\Gamma}(\phi(y)+S_c(y,x)),\qquad \text{for }x\in\Gamma,
    \end{equation*}
    we get, according to the comparison principle~\cite[Theorem 7.1]{Siconolfi22}, \cref{weakltb,Scalprop,maxsubsol}, that
    \begin{equation}\label{eq:ltb1}
        \begin{aligned}
            w(x)&\le(\Scal(t)w)(x)+ct\le u(x),&& \text{for any $(x,t)\in\Gamma\times\Rds^+$},\\
            w(x)&=(\Scal(t)w)(x)+ct=u(x),&& \text{for any $(x,t)\in\wtAcal_\Gamma\times\Rds^+$}.
        \end{aligned}
    \end{equation}

    Let us define the semilimits
    \begin{align}
        \unphi(x)\coloneqq\;&\sup\mleft\{\limsup_{n\to\infty}(\Scal(t_n)\phi)(x_n)+ct_n\mright\},\label{eq:limsuptsol1}\\
        \ovphi(x)\coloneqq\;&\inf\mleft\{\liminf_{n\to\infty}(\Scal(t_n)\phi)(x_n)+ct_n\mright\}\label{eq:limsuptsol2},
    \end{align}
    where the supremum and the infimum are taken over the sequences $\{x_n\}_{n\in\Nds}$ converging to $x$ and the positive diverging sequences $\{t_n\}_{n\in\Nds}$. Arguing as in~\cite[Proposition 5.13]{Pozza25_2} we get

    \begin{prop}\label{ovphisupsol}
        Given $\phi\in C(\Gamma)$, let $\unphi$ and $\ovphi$ be as defined in~\eqref{eq:limsuptsol1} and~\eqref{eq:limsuptsol2}, respectively. Then $\unphi$ and $\ovphi$ are a subsolution to \hrefc{} on $\Gamma$ and a supersolution to \hrefc{} on $\Gamma\setminus\wtAcal_\Gamma$, respectively.
    \end{prop}

    Defining the semilimits $\unw$ and $\ovw$ as in~\cref{eq:limsuptsol1,eq:limsuptsol2} with $\phi$ replaced by $w$, respectively, we get from~\cref{eq:ltb1}
    \begin{align*}
        \ovw(x)&\le\unw(x)\le u(x),\qquad \text{for any $x\in\Gamma$},\\
        \ovw(x)&=\unw(x)=u(x),\qquad \text{for any $x\in\wtAcal_\Gamma$},
    \end{align*}
    then \cref{ovphisupsol} and the comparison principle for the eikonal equation~\cite[Theorem 5.3]{Pozza25} yield that $\ovw=\unw=u$ on $\Gamma$, which in turn shows
    \begin{equation}\label{eq:ltb2}
        \lim_{t\to\infty}(\Scal(t)w)(x)+ct=u(x),\qquad \text{for every $x\in\Gamma$}.
    \end{equation}

    Next we have by \cref{Scalprop} that
    \begin{equation*}
        (\Scal(t)w)(x)\le(\Scal(t)\phi)(x),\qquad \text{for all $(x,t)\in\Gamma\times\Rds^+$},
    \end{equation*}
    therefore, taking into account~\cref{eq:ltb2},
    \begin{equation*}
        u(x)\le\ovphi(x)\le\unphi(x),\qquad \text{for every $x\in\Gamma$}.
    \end{equation*}
    Arguing as in~\cite[Lemmas 5.16 and 5.17]{Pozza25} we get that $\unphi=u$ on $\wtAcal_\Gamma$, thus another application of~\cite[Theorem 5.3]{Pozza25} proves \cref{ltb}.

    \printbibliography[heading=bibintoc]

\end{document}